\documentclass[11pt,a4paper]{amsart}
\usepackage[english]{babel}
\usepackage{color,psfrag}
\usepackage{times}
\usepackage{graphicx}
\usepackage{amscd}
\usepackage{amsmath}
\usepackage{amsfonts}
\usepackage{amssymb}
\usepackage{amsthm}
\usepackage{latexsym}
\newtheorem{theorem}{Theorem}[section]
\newtheorem{proposition}{Proposition}[section]
\newtheorem{lemma}{Lemma}[section]
\newtheorem{corollary}{Corollary}[section]

\newtheorem{definition}{Definition}[section]
\newtheorem{problem}{Problem}[section]

\def\neweq#1{\begin{equation}\label{#1}}
\def\endeq{\end{equation}}
\def\eq#1{(\ref{#1})}

\def\ds{\displaystyle}

\newcommand{\R}{\mathbb{R}}
\newcommand{\N}{\mathbb{N}}
\newcommand{\eps}{\varepsilon}

\newcommand{\SN}{{\mathbb S}^{n-1}}

\begin{document}

\title[Existence and stability properties of entire solutions to $(-\Delta)^m u=e^u$]{Existence and stability properties of entire solutions to the polyharmonic equation $(-\Delta)^m u=e^u$ for any $m\ge 1$}

\author[Alberto FARINA]{Alberto FARINA}
\address{\hbox{\parbox{5.7in}{\medskip\noindent{Universit\'e de Picardie Jules Verne,\\
Facult\'e des Sciences, \\
Rue Saint-Leu 33, 80039 Amiens, France. \\[3pt]
\em{E-mail address: }{\tt alberto.farina@u-picardie.fr}}}}}

\author[Alberto FERRERO]{Alberto FERRERO}
\address{\hbox{\parbox{5.7in}{\medskip\noindent{Universit\`a del
Piemonte Orientale ``Amedeo Avogadro'',\\
Dipartimento di Scienze e Innovazione Tecnologica, \\
        Viale Teresa Michel 11, 15121 Alessandria, Italy. \\[3pt]
        \em{E-mail address: }{\tt alberto.ferrero@mfn.unipmn.it}}}}}

\date{\today}



\begin{abstract} We study existence and stability properties of
entire solutions of a polyharmonic equation with an exponential
nonlinearity. We study existence of radial entire solutions and we
provide some asymptotic estimates on their behavior at infinity.
As a first result on stability we prove that stable solutions
(not necessarily radial) in dimensions lower than the conformal one never exist.
On the other hand, we prove that radial entire solutions which are stable
outside a compact always exist both in high and low dimensions. In
order to prove stability of solutions outside a compact set we
prove some new Hardy-Rellich type inequalities in low dimensions.
\end{abstract}

\maketitle

\section{Introduction}

We are interested in existence, nonexistence and stability properties of global
solutions for the polyharmonic equation
\begin{equation} \label{eq_1}
(-\Delta)^m u=e^u \qquad \text{in } \R^n \, .
\end{equation}
This problem is the natural extension to the polyharmonic case of
the Gelfand equation \cite{Gelfand}
\begin{equation} \label{eq:secondorder}
-\Delta u=e^u\quad\mbox{in}\quad \mathbb{R}^n,\qquad n\ge 1 \, .
\end{equation}
Equation \eqref{eq:secondorder} describes problems of thermal
self-ignition \cite{Gelfand}, diffusion phenomena induced by
nonlinear sources \cite{js} or a ball of isothermal gas in
gravitational equilibrium as proposed by lord Kelvin \cite{c}. For
results concerning properties of solutions of the Gelfand equation
in the whole $\R^n$ or in bounded domains see
\cite{BrezisVazquez,Dancer,farina,jl,MignotPuel} and the
references therein.

Recently, problem \eqref{eq_1} in the biharmonic case $m=2$ was
widely studied, see
\cite{agg,aggm,bffg,berchiogazzola,ChangChen,ddgm,dfg,dggw,Lin,w,wy}.
In the listed papers the biharmonic version of the Gelfand
equation was considered both in bounded domains with suitable
boundary conditions and in the whole $\R^n$; several questions
were tackled, from the existence of solutions to their qualitative
and stability properties. For other results concerning radial
entire solutions of nonlinear biharmonic equations see also
\cite{fergru,fergrukar,ferwar,gazgru,book,kara} and the references
therein.

The study of higher order elliptic equations like in \eqref{eq_1}
is motivated by the problem formulated by P.L. Lions \cite[Section
4.2 (c)]{li}, namely: {\it Is it possible to obtain a description
of the solution set for higher order semilinear equations
associated to exponential nonlinearities?}

Our paper is essentially focused on the existence and stability
properties of entire solutions of \eqref{eq_1}. This paper has the
purpose of being a first step in a deeper comprehension of
properties of entire solutions of \eqref{eq_1}. Throughout this
paper, by entire solution to problem \eqref{eq_1} we mean a
classical solution $u$ of the equation in \eqref{eq_1} which exists
for all $x\in\R^n$.

Concerning existence of entire solutions we describe in which way
existence of global radial solutions of \eqref{eq_1} is influenced
by the fact that $m$ is even or not. For results about radial
solutions of nonlinear polyharmonic equations see the papers
\cite{dls,ls} and the references therein.

In the present paper, in looking for radial solutions of
\eqref{eq_1}, we consider the following initial value problem
\begin{equation} \label{Cauchy}
\left\{
\begin{array}{ll}
(-\Delta)^m u(r)=e^{u(r)} \qquad \text{for } r\in [0,R(\alpha_0,\dots,\alpha_{m-1})) \\
u(0)=\alpha_0 \, , \quad u'(0)=0 \\
\Delta^k u(0)=\alpha_{k} \, , \quad (\Delta^k u)'(0)=0 \, \quad
\text{for any } k\in \{1,\dots,m-1\}
\end{array}
\right.
\end{equation}
where $\alpha_0,\dots,\alpha_{m-1}$ are arbitrary real numbers and
$R(\alpha_0,\dots,\alpha_{m-1})$ is the supremum of the maximal
interval of continuation of the corresponding solution. The
conditions $u'(0)=0$ and on $(\Delta^k u)'(0)=0$ are necessary for
having smoothness of the solution at the origin.

As a first result, we prove that in the case $m$ odd, for any
$\alpha_0,\dots,\alpha_{m-1}$ the corresponding solution of
\eqref{Cauchy} is an entire solution of \eqref{eq_1}, see Theorem
\ref{t:existence-odd}.  In dimension $n=1$ we also prove that all solutions of \eqref{eq_1}, not necessarily symmetric,
are global, see Theorem \ref{t:ex-odd}.

On the other hand, if $m$ is even and $n=1$
or $n=2$ then any solution of \eqref{Cauchy} is not global whereas
in dimension $n\ge 3$ both existence and nonexistence of global
solutions may occur. In this last situation we give a sufficient
and necessary condition for the existence of radial entire
solutions of \eqref{eq_1}, see Theorem \ref{t:existence-even}.
More precisely, this theorem shows that for any $\alpha_0\in\R$,
\eqref{Cauchy} admits a global solution if and only if the
$(m-1)$-tuple $(\alpha_1,\dots,\alpha_{m-1})$ belongs to a
suitable nonempty closed set depending on $\alpha_0$, denoted by
$\mathcal A_{\alpha_0}$. The nonexistence result in Theorem \ref{t:existence-even} (i) is extended for $n=1$ to all solutions
of \eqref{eq_1}, see Theorem \ref{t:non-ex-even}.

The second purpose of this paper is to shed some light on the
asymptotic behavior of global solutions of \eqref{Cauchy} as $r
\to +\infty$ and on their stability properties.

In this direction we first show in Proposition \ref{noboundbelow}
that all entire solutions of \eqref{eq_1} are unbounded from
below. In Theorem \ref{t:estimates} we restrict out attention to
radial solutions of \eqref{eq_1}. When $m$ is odd we prove that
for some special values of the initial conditions, problem
\eqref{Cauchy} admits solutions which blow down to $-\infty$ at
least as $r^4$ as $r\to +\infty$. Moreover if $1\le n\le 2m-1$ all
radial solutions of \eqref{eq_1} blow down to $-\infty$ at least
as a positive and integer power of $r$.

On the other hand when $m$ is even we prove that for any
$\alpha_0\in\R$, solutions corresponding to initial conditions
satisfying $(\alpha_1,\dots,\alpha_{m-1})\in \ds{\mathop{\mathcal
A_{\alpha_0}}^{\circ}}$, behave like $-C r^{2m-2}$ as $r\to
+\infty$ for a suitable constant $C>0$. If
$(\alpha_1,\dots,\alpha_{m-1})\in \partial \mathcal A_{\alpha_0}$
then the corresponding solution $u$ satisfies $u(r)=o(r^{2m-2})$
as $r\to +\infty$; however we are able to prove that $u$ blows
down to $-\infty$ at least as a logarithm as $r\to +\infty$, see
Theorem \ref{t:estimates} (iv). Such a logarithmic behavior
actually occurs when $m=2$ and $n\ge 5$, see \cite{agg}.
When $m\ge 3$ is any integer (possibly also odd) and $n=2m$ a
logarithmic behavior can be observed for a special class of
solutions of \eqref{eq_1}, see \cite{martinazzi} and the
references therein. More precisely, combining Theorems 1-2 in
\cite{martinazzi}, it can be shown that among solutions of
\eqref{eq_1} satisfying the condition $\int_{\R^{2m}}
e^{u}<+\infty$, the only ones which show a logarithmic behavior at
infinity are the explicit solutions given by
\begin{equation} \label{eq:sol-esplicita}
u(x)=2m \log \frac{2[(2m)!]^{\frac
1{2m}}\lambda}{1+\lambda^2|x-x_0|^2}
\end{equation}
where $\lambda>0$ and $x_0\in \R^{2m}$. For more details see also
Proposition \ref{p:martinazzi} and Corollary \ref{c:martinazzi} in
the present paper.

Then we focus our attention on stability and stability outside a
compact set. For a rigorous definition of these two notions see
Section \ref{s:stability}. In Theorem \ref{t:non-ex-stable} we
prove that \eqref{eq_1} admits no stable solutions (also non
radial) if $n$ is less or equal than the conformal dimension $2m$.
However, in Theorems \ref{t:staocs-odd}-\ref{t:staocs-even} we are
able to prove that if $3\le n\le 2m$ then \eqref{eq_1} admits radial
solutions which are stable outside a compact set. Moreover if
$m\ge 3$ is odd and $1\le n\le 2m-1$ then all radial solutions of
\eqref{eq_1} are stable outside a compact set, see Theorem \ref{t:n<=2m-1}.

In the supercritical dimensions $n>2m$ we prove that if $m$ is odd
then \eqref{eq_1} admits radial solutions that are stable outside
a compact set and if $m$ is even then, for any $\alpha_0\in\R$,
all solutions of \eqref{Cauchy} such that
$(\alpha_1,\dots,\alpha_{m-1})\in \ds{\mathop{\mathcal
A_{\alpha_0}}^{\circ}}$, are stable outside a compact set; the
question about the stability outside a compact set in the case
$(\alpha_1,\dots,\alpha_{m-1})\in \partial\mathcal A_{\alpha_0}$
is still open, see Problem \ref{prob:2} (ii).

The question about the existence of (globally) stable solutions is
completely open both in the cases $m$ odd and $m$ even, see
Problem \ref{prob:2} (iii). Let us try to explain the main
difficulties that one has to face in order to determine stability
of radial solutions. In the case $m=1$ a complete description of
stability and stability outside compact sets of solutions (also
non radial) of \eqref{eq_1} is available, see
\cite{Dancer,farina}. In the case $m=2$ a complete picture on
stability and stability outside compact sets was given in
\cite{bffg,dggw} at least for radial solutions.

If we look at radial solutions in the case $m=2$, we see that in
\cite{bffg} the authors are able to obtain asymptotic and global
estimates on solutions by exploiting a suitable change of
variables which reduces the ordinary differential equation in
\eqref{Cauchy} into a nonlinear fourth order autonomous equation,
see \cite[Proof of Lemma 12]{bffg}. In turn, this fourth order
autonomous equation may reduced to a dynamical system of four
first order equations. In this situation the dimension $n$ plays a
crucial role in determining stability properties of radial
solutions: indeed in dimension $n\ge 13$ the above mentioned
fourth order autonomous equation shows a non oscillatory behavior
of its solutions and this, combined with the classical
Hardy-Rellich inequality \cite{Rellich}, gives stability of
solutions; on the other hand in dimensions $5\le n<13$ ($n=4$ is
the critical dimension) the autonomous fourth order equation shows
an oscillatory behavior and this justifies the existence of radial
unstable solutions.

When we consider higher powers $m$ of $-\Delta$ the situation
seems to be quite different for the following reason. In a
completely similar way the ordinary differential equation in
\eqref{Cauchy} may be reduced to an autonomous equation of order
$2m$. But this time a non oscillatory behavior, similar to the one
observed in the case $m=2$ when $n\ge 13$, seems not to take place
also in large dimensions as one can see from Section
\ref{s:dynamical}.

A relevant part of this paper is devoted to a class of
Hardy-Rellich type inequalities when $n$ is less or equal than the
corresponding critical dimension. The first result in this
direction is Proposition \ref{p:Hardy-Rellich} which can be
obtained with an iterative procedure by using a result of
\cite{calmus}. The results in Theorems
\ref{t:preliminar-1}-\ref{t:Hardy-Rellich-2} are new and their
proofs are based on suitable Emden type transformation. This kind
of procedure was already used in \cite{calmus} in order to obtain
Hardy-Rellich type inequalities in conical domains.

This paper is organized as follows. In Section \ref{existbeh} we
state existence and nonexistence results for solutions to
\eqref{eq_1} and we provide some estimates on the asymptotic
behavior of its radial solutions as $|x|\to +\infty$. In Section
\ref{s:stability} we give some results about stability and
stability outside compact sets of solutions to
\eqref{eq_1}. To this end, we need some Hardy-Rellich inequalities
which are stated in Section \ref{s:Hardy-Rellich}. Sections
\ref{s:odd}-\ref{s:stability-proof} are devoted to the proofs of
the main results. In Section \ref{s:dynamical} we explain in which
way further results on radial solutions of \eqref{eq_1} could be
obtained by mean of a suitable change of variable and we present
some open questions. Finally in the appendix we state a couple of
well-known results dealing with continuous dependence on the
initial data and with a comparison principle.

\section{Existence and asymptotic behavior of radial entire solutions of
\eqref{eq_1}} \label{existbeh}

We start with following existence result for radial entire
solutions of \eqref{eq_1} in the case $m$ odd :

\begin{theorem} \label{t:existence-odd}
Let $n\ge 1$ and $m\ge 1$ odd. Then for any
$\alpha_0,\dots,\alpha_{m-1}\in \R$ problem \eqref{Cauchy} admits
a unique global solution.
\end{theorem}

In order to describe what happens in the case $m$ even we
introduce the following notation accordingly with
\cite{aggm,bffg}: we write $\alpha$ in place of $\alpha_0\in \R$
and we rename the numbers $\alpha_1,\dots,\alpha_{m-1}$
respectively $\beta_1,\dots,\beta_{m-1}$. Then we put
$\beta:=(\beta_1,\dots,\beta_{m-1})\in \R^{m-1}$ and we denote by
$u_{\alpha,\beta}$ the corresponding solution of \eqref{Cauchy}.
Finally for any $\alpha\in\R$ fixed, we introduce the set
\begin{equation} \label{eq:A-alpha}
\mathcal A_\alpha:=\{\beta\in\R^{m-1}:u_{\alpha,\beta} \ \text{is
a global solution of } \eqref{Cauchy}\} \, .
\end{equation}

We prove

\begin{theorem} \label{t:existence-even} Let $m\ge 2$ even and let $\mathcal A_\alpha$ be the set introduced
in \eqref{eq:A-alpha}. Then the following statements hold true:
\begin{itemize}
\item[(i)] if $n=1$ or $n=2$ then for any $\alpha\in \R$ the set
$\mathcal A_\alpha$ is empty.

\item[(ii)] if $n\ge 3$ then for any $\alpha\in\R$ the set
$\mathcal A_\alpha$ is nonempty and moreover there exists a
function

\noindent $\Phi_\alpha:\R^{m-2}\to (-\infty,0)$ such that
$$
\mathcal A_\alpha=\{\beta=(\beta_1,\dots,\beta_{m-1})\in
\R^{m-1}:\beta_{m-1}\le \Phi(\beta_1,\dots,\beta_{m-2}) \} \, ;
$$
\item[(iii)] if $n\ge 3$ then for any $\alpha\in\R$, $\Phi_\alpha$
is a continuous function, $\mathcal A_\alpha$ is closed, $\partial
\mathcal A_\alpha$ coincides with the graph of $\Phi_\alpha$ and
$$
\ds{\mathop{\mathcal
A_\alpha}^{\circ}}=\{\beta=(\beta_1,\dots,\beta_{m-1})\in
\R^{m-1}:\beta_{m-1}< \Phi(\beta_1,\dots,\beta_{m-2}) \} \, ;
$$
\item[(iv)] if $n\ge 3$ and $m\ge 4$ then for any $\alpha\in \R$,
$\Phi_\alpha$ is decreasing with respect to each variable i.e. the
map $t\mapsto
\Phi_\alpha(\beta_1,\dots,\beta_{k-1},t,\beta_{k+1},\dots,\beta_{m-2})$
is decreasing in $\R$ for any $k\in\{1,\dots,m-2\}$.
\end{itemize}
\end{theorem}

When $m=2$ the function $\Phi_\alpha$ introduced in the statement
of Theorem \ref{t:existence-even} is defined on the zero
dimensional space $\{0\}$ and the set $\mathcal A_\alpha$ becomes
$\{\beta\in \R:\beta\le \Phi_\alpha(0)\}$. The result in this
particular case was already obtained in \cite{agg}.

We observe that the nonexistence result proved in Theorem
\ref{t:existence-even} for $n=1$ remains valid also for any kind
of solutions of \eqref{eq_1}, also nonsymmetric:

\begin{theorem} \label{t:non-ex-even}
If $n=1$ and $m\ge 2$ is even then \eqref{eq_1} admits no entire solutions.
\end{theorem}

On the other hand, when $n=1$ and $m\ge 1$ is odd
we have
\begin{theorem} \label{t:ex-odd}
Let $n=1$ and $m\ge 1$ odd. Then any local solution of the
ordinary differential equation corresponding to \eqref{eq_1}, is
global. Moreover if $m=1$ then any solution of \eqref{eq_1} is
symmetric with respect to some point $x_0\in \R$. On the other
hand if $m\ge 3$, \eqref{eq_1} admits solutions which are not
symmetric with respect to any point $x_0\in \R$.
\end{theorem}

Next we provide some information on the qualitative behavior of
entire solutions of \eqref{eq_1}. First we show that any entire
solution (possibly non radial) of \eqref{eq_1} is not bounded from
below. Indeed if $u$ is a solution to \eq{eq_1} such that $u\geq
M$ for some $M\in\R$ then for any $q>1$ there exists $K(M,q)>0$
such that the inequality
$$
(-\Delta)^m u\geq K(M)|u|^q\quad \text{in } \R^n \,
$$
holds true. Then, from \cite[Theorem 4.1]{mp}, we infer
\begin{proposition}\label{noboundbelow}
For any $n\geq 1$ and $m\ge 1$, problem $\eq{eq_1}$ admits no
entire solutions bounded from below.
\end{proposition}

Then we deal with the asymptotic behavior of radial entire
solutions of \eqref{eq_1} as $|x|\to +\infty$. We prove

\begin{theorem} \label{t:estimates}
Let $n\ge 1$. Then the following statements hold
true:
\begin{itemize}
\item[(i)] if $m\ge 3$ is odd then for any solution $u$ of
\eqref{Cauchy} satisfying ${\rm sign} \, \alpha_k\neq (-1)^k$ for
at least one value of $k\in \{1,\dots,m-1\}$, we have
\begin{equation} \label{eq:quadratic}
u(r)<-C r^4 \qquad \text{for any }r>\overline r
\end{equation}
for some $C,\overline r>0$;

\item[(ii)] if $m\ge 3$ is odd and $1\le n\le 2m-1$ then any
solution $u$ of \eqref{Cauchy} satisfies
\begin{equation*} 
u(r)<-C r^K \qquad \text{for any } r>\overline r
\end{equation*}
for some $C,\overline r>0$ and $K$ positive integer;

\item[(iii)] if $n=1$ and $m\ge 1$ is odd then
any solution $u$ of \eqref{eq_1} (also nonsymmetric) satisfies
\begin{equation*} 
u(x)<-C |x|^K \qquad \text{for any } |x|>\overline r
\end{equation*}
for some $C,\overline r>0$ and $K$ positive integer;

\item[(iv)] if $m\ge 2$ is even, $\alpha\in\R$ and $\beta\in
\ds{\mathop{\mathcal A_\alpha}^{\circ}}$ then there exists $C>0$
such that
\begin{equation*}
u_{\alpha,\beta}(r)\sim -C r^{2m-2} \qquad \text{as } r\to +\infty
\end{equation*}
where we denoted by $u_{\alpha,\beta}$ the unique solution of
\eqref{Cauchy} corresponding to the couple $(\alpha,\beta)\in
\R^{m}$.

\item[(v)] if $m\ge 2$ is even, $\alpha\in\R$ and $\beta\in
\partial \mathcal A_\alpha$ then
$$
u_{\alpha,\beta}(r)=o(r^{2m-2}) \qquad \text{as } r\to +\infty
$$
and there exist $C,\overline r>0$ such that
\begin{equation*}
u_{\alpha,\beta}(r)<-2m \log r+C \qquad \text{for any }
r>\overline r
\end{equation*}
where we denoted by $u_{\alpha,\beta}$ the unique solution of
\eqref{Cauchy} corresponding to the couple $(\alpha,\beta)\in
\R^{m}$.
\end{itemize}
\end{theorem}

\begin{problem} {\rm As possible improvements of Theorem
\ref{t:estimates} we suggest the following open questions:
\begin{itemize}
\item[(i)] let $m\ge 3$ odd. Provide an upper estimate for all
radial solutions of \eqref{eq_1} and try to understand if for some
solution satisfying ${\rm sign} \, \alpha_k=(-1)^k$ for any $k\in
\{1,\dots,m-1\}$, \eqref{eq:quadratic} still holds true;

\item[(ii)] let $m\ge 4$ even. Determine the exact behavior of
radial solutions $u_{\alpha,\beta}$ of \eqref{eq_1} when $\beta\in
\partial \mathcal A_\alpha$.
\end{itemize}
}
\end{problem}

Something more precise about the behavior at infinity of entire
solutions (also non radial) of \eqref{eq_1} can be shown in the
conformal dimension $n=2m$ for any $m\ge 1$.

For example if $m$ is odd and $n=2m$, by taking $u$ in the form
\eqref{eq:sol-esplicita}, we see that there exist radial solutions
of \eqref{eq_1} which do not satisfy the estimate in Theorem
\ref{t:estimates} (i).

More generally we state the following results which are a quite
immediate consequence of Theorems 1-2 in \cite{martinazzi}.

\begin{proposition}\label{p:martinazzi}
Let $m\ge 2$ and $n=2m$. Let $u$ be a solution to \eqref{eq_1}
such that $e^u\in L^1(\R^{2m})$ and let
$\gamma:=\frac{1}{|{\mathbb S}^{2m}|(2m)!}
    \,\int_{\R^{2m}} e^{u}\,dx$ where $|{\mathbb S}^{2m}|$ denotes
    the surface measure of the $2m$-dimensional unit sphere in
    $\R^{2m+1}$.
     The following statements hold true:
\begin{itemize}
\item[$(i)$] the function $u$ can be represented as
\begin{equation*} 
u(x)=v(x)+p(x)
\end{equation*}
where $p$ is a polynomial bounded from above of degree at most
$2m-2$ and $v$ is a function satisfying
$$
v(x)=-2m\gamma \log|x|+o(\log|x|) \qquad \text{as } |x|\to +\infty
\, ;
$$
\item[$(ii)$] the function $u$ is of the form
\eqref{eq:sol-esplicita} if and only if $u(x)=o(|x|^2)$ as $|x|\to
+\infty$.
\end{itemize}
\end{proposition}

\begin{corollary}\label{c:martinazzi}
Let $m\ge 2$ even and let $n=2m$. Let $u$ be of the form
\eqref{eq:sol-esplicita} with $x_0=0$; let $\alpha:=u(0)$ and
$\beta\in \R^{m-1}$ be the corresponding initial values according
to the notation of Theorem \ref{t:existence-even}. Then $\beta \in
\partial \mathcal A_\alpha$.
\end{corollary}

\section{Stability properties of solutions to \eqref{eq_1}}
\label{s:stability}

We start with the definition of stability and stability outside a
compact set for entire solutions of \eqref{eq_1}. In the sequel,
for any open set $\Omega\subset \R^n$, we denote by
$C^\infty_c(\Omega)$ the set of $C^\infty$ functions whose support
is compactly included in $\Omega$.

\begin{definition} A solution $u\in C^{2m}(\R^n)$ to \eqref{eq_1}
is stable if
\begin{equation} \label{eq:stability}
\begin{array}{ll}
 \ds{\int_{\R^n} |\Delta^{m/2} \varphi|^2 dx-\int_{\R^n} e^u \varphi^2
dx}\ge 0 & \qquad \text{for any } \, \varphi \in
C^\infty_c(\R^n)\, , \ \ \text{if } m
\ \text{is even} \, , \\[12pt]
\ds{\int_{\R^n} |\nabla(\Delta^{\frac{m-1}2} \varphi)|^2
dx-\int_{\R^n} e^u \varphi^2 dx}\ge 0 & \qquad \text{for any } \,
\varphi \in C^\infty_c(\R^n) \, , \ \ \text{if } m \ \text{is odd}
\, .
\end{array}
\end{equation}
A solution $u\in C^{2m}(\R^n)$ to \eqref{eq_1} is stable outside a
compact set $K$ if
\begin{equation}
\begin{array}{ll} \label{eq:stability-ocs}
 \ds{\int_{\R^n} |\Delta^{m/2} \varphi|^2 dx-\int_{\R^n} e^u \varphi^2
dx}\ge 0 & \qquad \text{for any } \, \varphi \in
C^\infty_c(\R^n\setminus K)\, , \ \ \text{if } m
\ \text{is even} \, , \\[12pt]
\ds{\int_{\R^n} |\nabla(\Delta^{\frac{m-1}2} \varphi)|^2
dx-\int_{\R^n} e^u \varphi^2 dx}\ge 0 & \qquad \text{for any } \,
\varphi \in C^\infty_c(\R^n\setminus K) \, , \ \ \text{if } m \
\text{is odd} \, .
\end{array}
\end{equation}
\end{definition}

We state the following nonexistence result for stable (also non
radial) solutions of \eqref{eq_1} in dimension $n\le 2m$.

\begin{theorem} \label{t:non-ex-stable}
If $n\le 2m$ then \eqref{eq_1} admits no stable solutions.
\end{theorem}
In the case $n=1$ we prove stability outside a compact set of all
solutions of \eqref{eq_1}.

\begin{theorem} \label{t:n=1-socs} Let $n=1$ and let $m\ge 1$ be
odd. Then any solution of \eqref{eq_1} is stable outside a compact
set.
\end{theorem}

Next we state some results about the existence of
radial stable solutions of \eqref{eq_1} in both the cases $m$ odd
and $m$ even. We start with the following result valid for $m$ odd
and $n$ strictly below the conformal dimension $2m$.

\begin{theorem} \label{t:n<=2m-1} Let $m\ge 1$ odd and $1\le
n\le 2m-1$. Then any radial solution of \eqref{eq_1} is stable
outside a compact set.
\end{theorem}

\begin{theorem} \label{t:staocs-odd}
Let $n\ge 1$ and $m\ge 3$ odd. Let $u$ be a solution of
\eqref{Cauchy} satisfying ${\rm sign} \, \alpha_k\neq (-1)^k$ for
at least one value of $k\in \{1,\dots,m-1\}$. Then $u$ is a
solution of \eqref{eq_1} stable outside a compact set.
\end{theorem}

\begin{theorem} \label{t:staocs-even}
Let $n\ge 3$ and $m\ge 2$ even. Let $\alpha\in \R$ and let $\beta
\in \ds{\mathop{\mathcal A_\alpha}^{\circ}}$ with $\mathcal
A_\alpha$ ad in Theorem \ref{t:existence-even}. Let
$u_{\alpha,\beta}$ be the corresponding solution of
\eqref{Cauchy}. Then $u_{\alpha,\beta}$ is a solution of
\eqref{eq_1} stable outside a compact set.
\end{theorem}

Then we consider the case $n=2m$.

\begin{theorem} \label{t:conformal-d} Let $m\ge 1$ and $n=2m$. Let $u$ be the
solution defined in \eqref{eq:sol-esplicita}. Then $u$ is stable
outside a compact set.
\end{theorem}

As a last result of this section we resume in a  unique theorem all the previous statements proved in the case $n=1$:

\begin{theorem} \label{t:1-d} Let $n=1$.
\begin{itemize}
\item[(i)] If $m$ is even then \eqref{eq_1} admits no entire solutions.

\item[(ii)] If $m$ is odd then any local solution of \eqref{eq_1} is global.

\item[(iii)] If $m=1$ then all solutions of \eqref{eq_1} are symmetric with respect to some point while if $m\ge 3$ is odd then \eqref{eq_1} admits
entire solutions which are not symmetric with respect to any point.

\item[(iv)] If $m\ge 1$ is odd then all entire solutions of \eqref{eq_1} are unstable but are stable outside a compact set.
\end{itemize}
\end{theorem}

\bigskip

We want to emphasize the important role played by entire solutions of an
elliptic equation in the study of entire solutions
in higher dimensions. As one can see from Theorem \ref{t:1-d}, where in dimension $n=1$
we gave a complete description of properties of solutions of
\eqref{eq_1}, no stable solution exists for any $m\ge 1$; we have
in the case $m$ odd at most stability outside a compact set but
this property is not preserved after adding further dimensions.
Indeed if we consider a solution $u=u(x)$, $x\in\R$, of
\eqref{eq_1} stable outside a compact set (but unstable in view of
Theorem \ref{t:1-d} (iv)) and we see it as an entire solution of
\eqref{eq_1} in $\R^{k+1}$ then it becomes unstable outside any
compact set and in particular its Morse index is infinite. To see
this, take $\varphi\in C^\infty_c(\R)$ such that $\int_{\R}
[(\varphi^{(m)})^2-e^u\varphi^2] dx<0$, $\psi_1\in
C^\infty_c(\R^k)$, $\psi_1 \not\equiv 0$ and
$\psi_R(y):=\psi_1(y/R)$ for any $R>0$. Then one may check that
$$
\int_{\R^{k+1}}
|\nabla(\Delta^{\frac{m-1}2}(\varphi(x)\psi_R(y)))|^2 \, dxdy=R^k
\int_{\R^{k}} (\psi_1(y))^2 dy\cdot \int_\R
(\varphi^{(m)}(x))^2dx+o(R^k) \quad \text{as } R\to +\infty \, .
$$
Therefore
\begin{align*}
\int_{\R^{k+1}} &
\Big[|\nabla(\Delta^{\frac{m-1}2}(\varphi(x)\psi_R(y)))|^2-e^{u(x)}(\varphi(x))^2(\psi_R(y))^2\Big]
 \, dxdy \\
 & =R^k
\int_{\R^{k}} (\psi_1(y))^2 dy\cdot \int_\R
\left[(\varphi^{(m)}(x))^2-e^{u(x)}(\varphi(x))^2 \right] dx
 +o(R^k) \quad \text{as } R\to +\infty \, .
\end{align*}
For $R>0$ sufficiently large we have that the last line becomes
negative. Fixing such an $R>0$ and letting $\tau>0$, $\{e_1,\dots,e_{k+1}\}$ the standard basis in $\R^{k+1}$,
$v_\tau(x,y):=\varphi(x)\psi_R(y-\tau e_{j})\in C^\infty_c(\R^{k+1})$, $j\in \{2,\dots,k+1\}$,
we obtain
$$
\int_{\R^{k+1}}
[|\nabla(\Delta^{\frac{m-1}2}v_\tau)|^2-e^{u}v_\tau^2]
 \, dxdy<0 \qquad \text{for any } \tau>0 \, .
$$
This procedure may be extended to any unstable solution $u$ of a general problem in the form $(-\Delta)^m u=f(u)$
in $\R^n$ with $n\ge 1$ and $f\in C^1(\R)$.

\begin{problem} \label{prob:2}
{\rm Concerning stability properties of solutions
of \eqref{eq_1} we suggest the following questions:
\begin{itemize}
\item[(i)] Let $m\ge 3$ odd. Study stability outside a compact set
of radial solutions of \eqref{eq_1} satisfying ${\rm sign} \,
\alpha_k=(-1)^k$ for all $k\in \{1,\dots,m-1\}$.

\item[(ii)] Let $m\ge 4$ even. Study stability outside a compact
set of radial solutions $u_{\alpha,\beta}$ of \eqref{eq_1}
satisfying $\beta\in \partial\mathcal A_\alpha$. Only in the case
$n=2m$ we can conclude that such solutions are stable outside a
compact set. This follows immediately combining Corollary
\ref{c:martinazzi} and Theorem \ref{t:conformal-d}.

\item[(iii)] Let $m$ be any integer satisfying $m\ge 3$. Study
existence of radial entire solutions of \eqref{eq_1} which are
(globally) stable. See also the end of Section \ref{s:dynamical}
for more details.

\item[(iv)] Let $n=2$ and $m$ even. Study existence of entire
solutions of \eqref{eq_1}. We already know from Theorem
\ref{t:non-ex-even} that no radial entire solution exists.
Moreover no nonradial entire solution can be constructed by looking at solutions depending only on
one variable, see Theorem \ref{t:1-d}. We ask if entire solutions of \eqref{eq_1} really exist in this case.
\end{itemize}
}
\end{problem}

\section{Some higher order Hardy-Rellich type inequalities}
\label{s:Hardy-Rellich}

In this section we state some Hardy-Rellich type inequalities of
fundamental importance for determining stability outside compact
sets of solutions of \eqref{eq_1} especially in low dimensions.

Before these statements we recall from \cite{m} some higher order
classical Hardy-Rellich inequalities with optimal constants, see
also \cite{allegri,DH}. In the rest of this paper we put $\prod_{i=j}^k a_i=1$ whenever $k<j$.

\begin{proposition} \label{p:mitidieri} (\cite[Theorem 3.3]{m})
The following statements hold true:
\begin{itemize}
\item[(i)] if $k\ge 1$ and $n>4k$ then
\begin{equation*}
A_{n,k} \int_{\R^n} \frac{\varphi^2}{|x|^{4k}}\, dx\le \int_{\R^n}
|\Delta^k\varphi|^2 dx \qquad \text{for any } \varphi\in
C^\infty_c(\R^n)
\end{equation*}
where
\begin{equation*}
A_{n,k}:=\frac{1}{16^k} \prod_{i=0}^{k-1} (n-4k+4i)^2(n+4k-4i-4)^2;
\end{equation*}

 \item[(ii)] if $k\ge 0$ and $n>4k+2$ then
\begin{equation*}
B_{n,k} \int_{\R^n} \frac{\varphi^2}{|x|^{4k+2}}\, dx\le
\int_{\R^n} |\nabla(\Delta^k\varphi)|^2 dx \qquad \text{for any }
\varphi\in C^\infty_c(\R^n)
\end{equation*}
where
\begin{equation*}
B_{n,k}:=\frac{1}{16^k} \left(\frac{n-2}2\right)^2 \prod_{i=1}^{k}
(n-4i-2)^2(n+4i-2)^2
\end{equation*}
and moreover the constant $B_{n,k}$ is optimal in the case $k=0$.
\end{itemize}
\end{proposition}

The two inequalities stated in Proposition \ref{p:mitidieri} are
valid only for sufficiently large dimensions.

In order to obtain Hardy-Rellich type inequalities also in
low dimension we iterate inequality (0.6) in \cite{calmus} to
prove the following

\begin{proposition} \label{p:Hardy-Rellich}
Let $n\ge 2$ and let $k$ be a positive integer. Suppose that
$n\neq 2\ell$ for any $\ell\in \{1,\dots,2k\}$. For any $n\ge 2$
and any $\alpha\in \R$ define

\begin{equation} \label{mu-n-alpha}
\mu_{n,\alpha}:=\min_{j\in
\N\cup\{0\}}|\gamma_{n,\alpha}+j(n-2+j)|^2
\end{equation}
and
\begin{equation*}
\gamma_{n,\alpha}:=\left(\frac{n-2}2\right)^2-\left(\frac{\alpha-2}2\right)^2
\, .
\end{equation*}
Then we have
\begin{align} \label{eq:Hardy-Rellich-1}
\left(\prod_{i=1}^k \mu_{n,\alpha_i}  \right)  \int_{\R^n}
\frac{\varphi^2}{|x|^{4k}} \, dx\le \int_{\R^n}
|\Delta^k\varphi|^2 dx \qquad \text{for any } \varphi\in
C^\infty_c(\R^n\setminus\{0\})
\end{align}
and
\begin{align} \label{eq:Hardy-Rellich-2}
\left(\frac{n-2}2\right)^2 \left(\prod_{i=1}^k \mu_{n,\alpha_i}
\right) \int_{\R^n} \frac{\varphi^2}{|x|^{4k+2}} \, dx\le
\int_{\R^n} |\nabla(\Delta^k\varphi)|^2 dx \qquad \text{for any }
\varphi\in C^\infty_c(\R^n\setminus\{0\})
\end{align}
where we put $\alpha_i=-4k+4i$ for any $i\in \{1,\dots,k\}$.
\end{proposition}

In Proposition \ref{p:Hardy-Rellich} we excluded the case $n=1$
since we recall that in such a case the following inequalities
hold
\begin{proposition} \label{p:HR-1d} Let $\alpha\in \R$. Then for any
$\varphi\in C^\infty_c(\R\setminus\{0\})$ we have
\begin{equation} \label{eq:H-1d-0}
\frac{(\alpha-1)^2}{4}\int_{\R} |x|^{\alpha-2} \varphi^2 \, dx\le
\int_{\R} |x|^{\alpha} |\varphi'|^2 dx \, .
\end{equation}
Applying \eqref{eq:H-1d-0} twice we also obtain
\begin{equation} \label{eq:HR-1d}
\frac{(\alpha-1)^2(\alpha-3)^2}{16}\int_{\R} |x|^{\alpha-4}
\varphi^2 \, dx\le \int_{\R} |x|^{\alpha} |\varphi''|^2 dx  \qquad
\text{for any } \varphi\in C^\infty_c(\R\setminus\{0\})\, .
\end{equation}
Moreover iterating \eqref{eq:HR-1d} and using the classical Hardy
inequality in dimension $n=1$, for any integer $k\ge 0$, we obtain
\begin{equation} \label{eq:HR-1d-ter}
2^{-4k-2} \left(\prod_{i=0}^{k-1} (4i-3)^2(4i-5)^2\right) \int_\R
\frac{\varphi^2}{|x|^{4k+2}} \, dx\le \int_\R |\varphi^{(2k+1)}|^2
dx \qquad \text{for any } \varphi\in C^\infty_c(\R\setminus\{0\})
\end{equation}
with ${\ds \prod_{i=0}^{k-1} (4i-3)^2(4i-5)^2}=0$ when $k=0$.
\end{proposition}

We observe that the constant $\prod_{i=1}^k \mu_{n,\alpha_i}$
appearing in \eqref{eq:Hardy-Rellich-1}-\eqref{eq:Hardy-Rellich-2}
is strictly positive under the assumptions of Proposition
\ref{p:Hardy-Rellich}. On the other hand, if $n=2\ell$ for some
$\ell\in \{1,\dots,2k\}$ then $ \prod_{i=1}^k \mu_{n,\alpha_i}=0$
making estimates \eqref{eq:Hardy-Rellich-1} and
\eqref{eq:Hardy-Rellich-2} trivial. In order to show this, it is
sufficient to observe that $\mu_{n,\alpha_i}=0$ if and only if
$1\le i\le \min\{k,k+1-\frac \ell 2\}$; moreover the minimum in
\eqref{mu-n-alpha} is achieved for $j=2k-\ell-2(i-1)$.

For the above mentioned reasons, we need a new Hardy-Rellich type
inequality which is meaningful also in dimensions satisfying
$n=2\ell$ for some $\ell\in \{1,\dots,2k\}$.

In the rest of the paper we denote by $B_R$ the ball in $\R^n$ of
radius $R>0$ centered at the origin. We start with the following
second order inequality with logarithmic weights:

\begin{theorem} \label{t:preliminar-1}
Let $n\ge 2$, $\alpha\le 0$ and $\beta\ge 0$. Let $\mu_{n,\alpha}$
and $\gamma_{n,\alpha}$ be as in Proposition \ref{p:Hardy-Rellich}
and suppose that $\mu_{n,\alpha}=0$. Then there exists $R>1$ large
enough such that
\begin{equation*}
2\bar \gamma_{n,\alpha} \left(\frac{\beta+1}2\right)^2
\int_{\R^n\setminus \overline B_R}
\frac{|x|^{\alpha-4}\varphi^2}{(\log|x|)^{\beta+2}} \, dx \le
\int_{\R^n\setminus \overline B_R} \frac{|x|^{\alpha} |\Delta
\varphi|^2}{(\log|x|)^\beta} \, dx \qquad \text{for any }
\varphi\in C^\infty_c(\R^n\setminus \overline B_R)
\end{equation*}
with
$\bar\gamma_{n,\alpha}:=\left(\frac{n-2}2\right)^2+\left(\frac{\alpha-2}2\right)^2$.
\end{theorem}

Iterating Theorem \ref{t:preliminar-1} we obtain the following

\begin{theorem} \label{t:Hardy-Rellich-2}
Let $k$ be a positive integer and let $n=2\ell$ for some $\ell\in
\{1,\dots,2k\}$. Let $\bar\gamma_{n,\alpha}$ be as in Theorem
\ref{t:preliminar-1}. Then there exists $R>1$ large enough such
that
\begin{equation} \label{eq:ineq-1}
2^k \left(\prod_{i=0}^{k-1} \bar \gamma_{n,-4i} \right)\cdot
\left(\prod_{i=0}^{k-1} \Big(\frac{2i+1}2\Big)^2 \right)
 \int_{\R^n\setminus \overline B_R}
\frac{\varphi^2}{|x|^{4k}(\log|x|)^{2k}} \, dx \le
\int_{\R^n\setminus \overline B_R} |\Delta^k \varphi|^2 dx
\end{equation}
and
\begin{equation} \label{eq:ineq-2}
2^{k-2} \left(\prod_{i=0}^{k-1} \bar \gamma_{n,-4i-2} \right)\cdot
\left(\prod_{i=0}^{k-1} \Big(\frac{2i+3}2\Big)^2 \right)
 \int_{\R^n\setminus \overline B_R}
\frac{\varphi^2}{|x|^{4k+2}(\log|x|)^{2k+2}} \, dx \le
\int_{\R^n\setminus \overline B_R} |\nabla(\Delta^k \varphi)|^2 dx
\end{equation}
for any $\varphi\in C^\infty_c(\R^n\setminus \overline B_R)$.
\end{theorem}

We observe that \eqref{eq:ineq-1}-\eqref{eq:ineq-2} may be
improved by using (0.6) in \cite{calmus} whenever the numbers
$\mu_{n,\alpha_i}$ with $\alpha_i=-4k+4i$ are strictly positive
and using Theorem \ref{t:preliminar-1} whenever they vanish.

\section{The case $m$ odd} \label{s:odd}

In this section we concentrate our attention on the case $m\ge 3$
odd being the existence of global radial solutions in the case
$m=1$ completely known, see for example \cite{jl}.

\begin{lemma}\label{odd}
Let $n\ge 1$, let $m\ge 3$ be odd and let $u$ be a solution of
\eqref{Cauchy} defined on the maximal interval of continuation
$[0,R(\alpha_0,\dots,\alpha_{m-1}))$. Then for any
$\alpha_0,\dots,\alpha_{m-1}\in \R$ we have that
$R(\alpha_0,\dots,\alpha_{m-1})=+\infty$.
\end{lemma}

\begin{proof}
Since $m\ge 3$ is odd we may write $\Delta(\Delta^{m-1}u)=-e^u$ so
that
\begin{equation} \label{div_1}
(r^{n-1}(\Delta^{m-1}u(r))')'=-r^{n-1}e^u \, .
\end{equation}
This shows that the map $r\mapsto r^{n-1} (\Delta^{m-1} u(r))'$ is
decreasing and since it equals to zero at $r=0$ then
\begin{equation} \label{eq_negative}
(\Delta^{m-1}u(r))'<0 \quad \text{for any }
r\in(0,R(\alpha_0,\dots,\alpha_{m-1}))\, .
\end{equation}
In particular the map $r\mapsto \Delta^{m-1}u(r)$ is decreasing
and hence
\begin{equation} \label{Delta_m-1}
\Delta^{m-1}u(r)\leq \alpha_{m-1} \qquad \text{for any }
r\in[0,R(\alpha_0,\dots,\alpha_{m-1})) \, .
\end{equation}
Consider now the unique (global) solution $w$ of the initial value
problem
\begin{equation} \label{Deltaw}
\left\{
\begin{array}{ll}
\Delta^{m-1} w(r)=\alpha_{m-1} \qquad r\in(0,+\infty) \\
w(0)=u(0) \, , \quad w'(0)=u'(0)=0 \\
\Delta^k w(0)=\Delta^k u(0) \, , \quad (\Delta^k w)'(0)=(\Delta^k
u)'(0)=0  \quad \text{for any } k\in\{1,\dots,m-2\} \, .
\end{array}
\right.
\end{equation}
By \eqref{Delta_m-1}, \eqref{Deltaw} and Proposition \ref{mckenna
reichel} we deduce that for any $r\in
[0,R(\alpha_0,\dots,\alpha_{m-1}))$
\begin{align*}
 & u(r)\leq w(r)  \, \quad u'(r)\leq w'(r) \\
& \Delta^k u(r)\leq \Delta^k w(r) \, , \quad (\Delta^k u(r))'\leq
(\Delta^k w(r))'\, , \quad \text{for all } k\in \{1,\dots,m-2\} \,
.
\end{align*}
If we now assume by contradiction that
$R(\alpha_0,\dots,\alpha_{m-1})<+\infty$ then $u$ would be bounded
from above in the interval $(0,R(\alpha_0,\dots,\alpha_{m-1}))$
and hence $e^u$ would be bounded in
$(0,R(\alpha_0,\dots,\alpha_{m-1}))$. After successive
integrations of the equation in \eqref{Cauchy}, one can prove that
$u$ and all its derivatives until order $2m-1$ are bounded. By a
standard argument from the theory of ordinary differential
equations it follows that $R(\alpha_0,\dots,\alpha_{m-1})=+\infty$
thus producing a contradiction. This completes the proof of the
lemma.
\end{proof}

\begin{lemma} \label{l:neg}
Let $n\ge 1$, let $m\ge 3$ be odd and let $u$ be a solution of
\eqref{Cauchy} defined on the maximal interval of continuation
$[0,+\infty)$. Then
\begin{equation} \label{eq_lim}
\lim_{r\to +\infty} \Delta^{k} u(r)\in [-\infty,0] \, .
\end{equation}
for any $k\in \{1,\dots,m-1\}$.
\end{lemma}

\begin{proof} The existence of the limit in \eqref{eq_lim} follows from \eqref{eq_negative} and an iterative procedure of
integration. Suppose by contradiction that there exists $\overline
k\in \{1,\dots,m-1\}$ such that
\begin{equation}
\label{elle>0} \ell_1:=\lim_{r\to +\infty} \Delta^{\overline k}
u(r)> 0 \, .
\end{equation}
Then there exists $\overline r>0$ such that
\begin{equation*}
 \Delta^{\overline k} u(r)>\frac{\ell_1}2   \qquad \text{for all } r>\overline r  \, .
\end{equation*}
After a couple of integrations one obtains
\begin{equation*}
\Delta^{\overline k-1} u(r)>\frac{\ell_1}{4n} r^2+o(r^2) \qquad
\text{as } r\to +\infty
\end{equation*}
and in particular $\lim_{r\to +\infty} \Delta^{\overline k-1}
u(r)=+\infty$. Iterating this procedure we arrive to prove that
\begin{equation*}
\lim_{r\to +\infty} u(r)=+\infty \, .
\end{equation*}
From this and \eqref{div_1} we deduce that for any $M>0$ there
exists $R_M>0$ such that
\begin{equation*}
(r^{n-1} (\Delta^{m-1}u(r))')'<-Mr^{n-1} \qquad \text{for all }
r>R_M \, .
\end{equation*}
After integration this produces
\begin{equation*}
\Delta^{m-1}u(r)<-\frac{M}{2n} r^2+o(r^2) \qquad \text{as } r \to
+\infty \, .
\end{equation*}
After a finite number of integrations we deduce that $\lim_{r\to
+\infty} \Delta^{\overline k}u(r)=-\infty$ a contradiction.
\end{proof}

We ask if \eqref{Cauchy} admits a solution $u$ for which the limit
in \eqref{eq_lim} can be strictly negative at least for one value
of $k\in \{1,\dots,m-1\}$. To this purpose we prove the following
\begin{lemma} \label{l:lim-0}
Let $n\ge 1$, let $m\ge 3$ be odd and let $u$ be a solution of
\eqref{Cauchy} defined on the maximal interval of continuation
$[0,+\infty)$. Suppose that
\begin{equation} \label{eq_lim-2}
\lim_{r\to +\infty} \Delta^{k} u(r)=0
\end{equation}
for any $k\in \{1,\dots,m-1\}$ even. Then
\begin{equation*}
{\rm sign} \, \alpha_k=(-1)^k  \qquad \text{for any } k\in
\{1,\dots,m-1\}  \, .
\end{equation*}
\end{lemma}

\begin{proof} By \eqref{Cauchy} we deduce that the map $r\mapsto
r^{n-1} (\Delta^{m-1} u)'(r)$ is decreasing in $[0,+\infty)$ and
since it vanishes at $r=0$ then it is negative in $(0,+\infty)$.
This implies that the map $r\mapsto \Delta^{m-1} u(r)$ is
decreasing in $(0,+\infty)$ and hence by \eqref{eq_lim-2} we have
that $\Delta^{m-1}u(r)>0$ for any $r\ge 0$. But
$(r^{n-1}(\Delta^{m-2}u)'(r))'=r^{n-1}\Delta^{m-1}u(r)>0$ and
hence the map $r\mapsto r^{n-1} (\Delta^{m-2} u)'(r)$ is
increasing in $[0,+\infty)$ and since it vanishes at $r=0$ then it
is positive in $(0,+\infty)$. This implies that the map $r\mapsto
\Delta^{m-2} u(r)$ is increasing in $(0,+\infty)$ and hence by
\eqref{eq_lim} we have that $\Delta^{m-2}u(r)<0$ for any $r\ge 0$.
Iterating this procedure we infer that for any $k\in
\{1,\dots,m-1\}$, $(-1)^k \Delta^k u(r)>0$ for any $r\ge 0$. In
particular by \eqref{Cauchy} we deduce that ${\rm sign} \,
\alpha_k=(-1)^k$ for any $k\in \{1,\dots,m-1\}$.
\end{proof}

As a consequence of Lemma \ref{l:lim-0} we prove the existence of
solutions of \eqref{Cauchy} satisfying a suitable estimate from
above.

\begin{lemma} \label{l:lim-0-cons}
Let $n\ge 1$, let $m\ge 3$ be odd and let $u$ be a solution of
\eqref{Cauchy} defined on the maximal interval of continuation
$[0,+\infty)$ and suppose that there exists $\overline k\in
\{1,\dots,m-1\}$ such that ${\rm sign}\,  \alpha_{\overline k}\neq
(-1)^{\overline k}$. Then there exist $C,\overline r>0$ such that
\begin{equation*}
u(r)<-C r^4 \qquad \text{for any } r>\overline r \, .
\end{equation*}
\end{lemma}

\begin{proof}
Let $\overline k$ be as in the statement. Then by Lemmas
\ref{l:neg}-\ref{l:lim-0}, we deduce that at least for one $k\in
\{1,\dots,m-1\}$ even we have that $\lim_{r\to +\infty} \Delta^k
u(r)<0$. Using this information and integrating we conclude that
${\ds \lim_{r\to +\infty} \Delta^2 u(r)}$ is strictly negative and
in particular there exist $\overline r,C>0$ such that
\begin{equation*}
\Delta^2 u(r)<-C \qquad \text{for any } r>\overline r \, .
\end{equation*}
After four integrations the conclusion of the lemma follows.
\end{proof}

We provide an estimate from above at infinity in the case $n\le
2m-1$.

\begin{lemma} \label{l:55} Let $m\ge 3$ be odd, let $1\le n\le 2m-1$ and let $u$ be a solution of
\eqref{Cauchy} defined on the maximal interval of continuation
$[0,+\infty)$. Then there exist a positive integer $K$ and
constants $C,\overline r>0$ such that
\begin{equation*}
u(r)<-Cr^K \qquad \text{for any } r>\overline r \, .
\end{equation*}
\end{lemma}

\begin{proof}
By Lemma \ref{l:neg} we know that only the two following
alternatives may occur: either there exists $\overline k\in
\{1,\dots,m-1\}$ such that
\begin{equation} \label{eq:alternative-1}
\lim_{r\to +\infty} \Delta^{\overline k} u(r)<0
\end{equation}
or
\begin{equation} \label{eq:alternative-2}
\lim_{r\to +\infty} \Delta^{j} u(r)=0 \qquad \text{for any } j\in
\{1,\dots,m-1\} \, .
\end{equation}
We divide the proof in three parts.

{\bf The case $n=1,2$.} Put $v=\Delta^{m-1} u$ so that $v$ is a
radial superharmonic function in $\R^n$. In particular the map
$r\mapsto r^{n-1}v'(r)$ is decreasing and it is also negative for
any $r>0$ being equal to zero at $r=0$. Hence
$$
r^{n-1} v'(r)<v'(1)<0 \qquad \text{for any } r>1 \, .
$$
Integrating we then obtain
\begin{equation*}
v(r)<
\begin{cases}
v(1)-|v'(1)|\log r & \qquad \text{for any } r>1 \ \text{if} \ n=2 \\
v(1)-|v'(1)|(r-1) & \qquad \text{for any } r>1 \ \text{if}  \ n=1
\, .
\end{cases}
\end{equation*}
In both cases $\lim_{r\to +\infty} v(r)=-\infty$. This implies
that there exist $C,\overline r>0$ such that $\Delta^{m-1}
u(r)<-C$ for any $r>\overline r$. The proof of the lemma follows
after an iterative procedure of integration.

{\bf The case $3\le n\le 2m-2$.} We prove that
\eqref{eq:alternative-1} holds true. Suppose by contradiction that
\eqref{eq:alternative-2} holds true. Then by \eqref{Cauchy} and
\eqref{eq:alternative-2} we have
\begin{align} \label{eq:sign-altern}
((-\Delta)^j u)'(r)<0 \, , \qquad (-\Delta)^j u(r)>0 \qquad
\text{for any } r>0 \ \text{and} \ j\in\{1,\dots,m-1\} \, .
\end{align}
Since $n\ge 3$ we may fix $k\in \{1,\dots,m-2\}$ such that
$2k=n-2$ if $n$ is even and $2k=n-1$ if $n$ is odd. For any $r\ge
0$ put $v(r)=(-\Delta)^{m-k} u(r)$. Then by \eqref{Cauchy} we have
$(-\Delta)^{k} v=(-\Delta)^m u>0$ so that $v$ is a radial
$k$-superpolyharmonic function in $\R^n$. For any $\eps>0$ we
introduce the function $w_\eps(r):=v(r)-\eps \Phi(r)$ defined for
any $r>0$ where $\Phi(r):=r^{2k-n}$ is up to a constant multiplier
the fundamental solution of $(-\Delta)^{k}$. In particular we have
that $(-\Delta)^{k} w_\eps=(-\Delta)^{k} v>0$ in $(0,+\infty)$.
Exploiting \eqref{eq:sign-altern} we deduce that it is not
restrictive to fix $\eps>0$ small enough in such a way that
\begin{align} \label{eq:sign-alt-bis}
& \left(r^{n-1} ((-\Delta)^j
w_\eps)'(r)\right)_{|r=1}=(-1)^j[(\Delta^j v)'(1)-\eps (\Delta^j
\Phi)'(1)]\\
\notag & \qquad =((-\Delta)^{m-k+j} u)'(1)+\eps(-1)^{j+1}
(\Delta^j \Phi)'(1)<0 \qquad \text{for any } j\in \{0,\dots,k-1\}\,
.
\end{align}
where by $(-\Delta)^0$ we simply mean the identity operator.

Since $(-\Delta)^{k} w_\eps>0$ then the map $r\mapsto r^{n-1}
((-\Delta)^{k-1} w_\eps)'(r)$ is decreasing in $(0,+\infty)$ and
its value at $r=1$ is negative in view of \eqref{eq:sign-alt-bis}.
This implies that $((-\Delta)^{k-1} w_\eps)'(r)<0$ for any $r>1$
and, in turn, that the map $r\mapsto (-\Delta)^{k-1} w_\eps$ is
decreasing in $(1,+\infty)$. But by \eqref{eq:alternative-2} and
the definition of $w_\eps$ we have that
$$
\lim_{r\to +\infty} (-\Delta)^{k-1} w_\eps(r)=0
$$
and hence $(-\Delta)^{k-1} w_\eps(r)>0$ for any $r>1$. Iterating
this procedure we deduce that for any $j\in \{1,\dots,k\}$,
$(-\Delta)^j w_\eps>0$ in $(1,+\infty)$ and $w_\eps>0$ in the same
interval. By definition of $v$ and $w_\eps$ we infer
\begin{equation} \label{eq:26}
(-\Delta)^{m-k} u(r)>\eps r^{2k-n} \qquad \text{for any } r>1 \, .
\end{equation}
After an iterative procedure of integration it follows that there
exist $C,\overline r>0$ such that
\begin{equation*} 
|\Delta u(r)|>Cr^{-n+2m-2} \qquad \text{for any } r>\overline r \, .
\end{equation*}
Actually in the case $n$ even we also have that $|\Delta
u(r)|>Cr^{-n+2m-2}\log r$ for any large $r$. However, in any case
we have that $\lim_{r\to +\infty} \Delta u(r)\neq 0$ in
contradiction with \eqref{eq:alternative-2}. We proved the
validity of \eqref{eq:alternative-1} and then the conclusion of
the lemma follows after an iterative procedure of integration.

{\bf The case $n=2m-1$.} If \eqref{eq:alternative-1} holds true
then the proof of the lemma follows after an iterative procedure
of integration. If \eqref{eq:alternative-2} holds true then we
proceed exactly as in the case $3\le n\le 2m-2$ until
\eqref{eq:26} that becomes $\Delta u(r)<-\eps r^{-1}$ for any
$r>1$. Then a couple of integrations shows that $u(r)<-C r$ for
any large $r$. This completes the proof also in this case.
\end{proof}

We conclude this section with an estimate from
above at infinity for solutions of \eqref{eq_1} when $n=1$.
\begin{lemma} \label{l:1-d} Let $m\ge 1$ be odd and let $n=1$. Let $u$ be a solution of
\eqref{eq_1}. Then there exist a positive integer $K$ and
constants $C,\overline r>0$ such that
\begin{equation} \label{eq:st-dec}
u(x)<-C|x|^K \qquad \text{for any } |x|>\overline r \, .
\end{equation}
\end{lemma}

\begin{proof} Consider first the case $m=1$. We claim that there exists
$x_0\in \R$ such that $u'(x_0)=0$. Suppose by contradiction that
$u'(x)\neq 0$ for any $x\in \R$. Up to replace $u$ with the
function $u(-x)$ we may assume that $u'(x)>0$ for any $x\in \R$ so
that $u$ is increasing. Hence there exist $C,M>0$ such that
$e^{u(x)}>C$ for any $x>M$. This shows that $u''<-C$ in
$(M,+\infty)$ so that ${\ds \lim_{x\to +\infty} u'(x)=-\infty}$, a
contradiction. This completes the proof of the claim. The conclusion of the proof follows since $u$ is strictly concave.

We divide the proof of the case $m\ge 3$ odd into two steps.

{\bf Step 1.} Let $k\in \{1,\dots,2m-3\}$ be odd
and assume that there exists $x_0\in\R$ such that
$u^{(k)}(x_0)=0$. We prove that at least one of the two
alternatives holds true: either \eqref{eq:st-dec} holds true for
some $C,\overline r,K$ or $u^{(k+2)}$ vanishes at some point.

Suppose that \eqref{eq:st-dec} does not hold true
for any possible choice of $C,\overline r,K$ and let us prove the
validity of the second alternative. Suppose by contradiction that
$u^{(k+2)}(x)\neq 0$ for any $x\in\R$. Up to replace $u$ with the
function $u(-x)$ we may assume that $u^{(k+2)}(x)>0$ for any
$x\in\R$. Then $u^{(k)}$ is strictly convex and since
$u^{(k)}(x_0)=0$, only two situations may occur:

{\bf Case 1.} ${\ds \lim_{x\to +\infty}
u^{(k)}(x)=+\infty}$;

{\bf Case 2.} ${\ds\lim_{x\to +\infty} u^{(k)}(x)<0}$.

We may exclude Case 1. Indeed, after a finite
number of integrations we would have ${\ds\lim_{x\to +\infty}
\!\!u(x)\!=\!+\infty}$ and hence
$$
\lim_{x\to +\infty} u^{(2m)}(x)=-\lim_{x\to +\infty}
e^{u(x)}=-\infty \, ;
$$
after a finite number of integrations we arrive to ${\ds\lim_{x\to
+\infty} u^{(k)}(x)=-\infty}$ a contradiction.

This means that only Case 2 may occur. But from
strict convexity we necessarily have ${\ds \lim_{x\to -\infty}\!\!
u^{(k)}(x)\!=\!+\infty}$. Combining these two informations, integrating
a finite number of times and taking into account that $k$ is odd,
we conclude that \eqref{eq:st-dec} holds true, a contradiction.

{\bf Step 2.} In this step we complete the proof
of the lemma. We may proceed by contradiction assuming that
\eqref{eq:st-dec} does not hold true for any possible choice of
$C,\overline r,K$. We claim that there exists $x_0\in \R$ such
that $u'(x_0)=0$. Proceeding by contradiction, up to replace $u$
with the function $u(-x)$, we may assume that $u'(x)>0$ for any
$x\in \R$. Therefore $u$ is increasing and hence $e^{u}$ is
bounded away from zero at $+\infty$. Then by \eqref{eq_1} we infer
that ${\ds \lim_{x\to +\infty} u^{(2m)}(x)<0}$ and after a finite
number of integrations we obtain ${\ds\lim_{x\to +\infty}
u'(x)=-\infty}$, a contradiction. Therefore, we may apply
inductively Step 1 and prove that for any $k\in \{1,\dots,2m-3\}$
only the second alternative my occur. In particular this shows
that $u^{(2m-1)}$ vanishes somewhere. But by \eqref{eq_1} we
deduce that $u^{(2m-1)}$ is decreasing and hence it is bounded
away from zero both at $+\infty$ and $-\infty$; more precisely
negative at $+\infty$ and positive at $-\infty$. Taking into
account that $k$ is odd, after a finite number of integrations the
validity of \eqref{eq:st-dec} follows.
\end{proof}

\section{The case $m$ even}
Since \eqref{eq_1} is invariant under the following transformation
$$
u_\lambda(x)=u(\lambda x)+2m\log\lambda \, , \qquad \lambda>0 \, ,
$$
up to fix the value $\alpha_0$, the behavior of solutions of
\eqref{Cauchy} only depends on the values of the parameters
$\alpha_1,\dots,\alpha_{m-1}$. For this reason it is convenient to
treat the real parameter $\alpha_0$ and the vector valued
parameter $(\alpha_1,\dots,\alpha_{m-1})$ in two different ways.

According to \cite{agg,bffg}, for any $\alpha\in\R$ and $\beta\in
\R^{m-1}$, let us denote by $u_{\alpha,\beta}$ the unique local
solution of \eqref{Cauchy} corresponding to $\alpha_0=\alpha$ and
$\alpha_k=\beta_k$ for any $k\in \{1,\dots,m-1\}$, and by
$R_{\alpha,\beta}$ the corresponding maximal interval of
continuation. Finally for any $\alpha\in\R$, we define the set
\begin{equation*}
\mathcal A_\alpha:=\{\beta\in\R^{m-1}:u_{\alpha,\beta} \ \text{is
a global solution of } \eqref{Cauchy}\} \, .
\end{equation*}

We first prove that in dimensions $n=1,2$ the set $\mathcal
A_\alpha$ is empty for any $\alpha\in\R$. In other words for $n=1$
and $n=2$ problem \eqref{eq_1} does not admit any entire radial
solution.

\begin{lemma} \label{l:empty}
Let $n=1$ or $n=2$ and $m$ even. Then for any $\alpha\in\R$ the
set $\mathcal A_\alpha$ is empty.
\end{lemma}

\begin{proof} 
Let $u$ be a solution of \eqref{Cauchy} and let $[0,R)$ its
maximal interval of continuation. Assume by contradiction that $u$
is such that $R=+\infty$. By \eqref{Cauchy}, we have that the
function $r\mapsto r^{n-1} (\Delta^{m-1}u)'(r)$ is increasing in
$[0,R)$ and hence there exists $C>0$ such that for any $r\ge 1$
\begin{equation*}
\Delta^{m-1}u(r)\ge
\begin{cases}
\Delta^{m-1}u(1)+C(r-1) & \qquad \text{if } n=1 \\[8pt]
\Delta^{m-1}u(1)+C\log r & \qquad \text{if } n=2 \, .
\end{cases}
\end{equation*}
After an iterative procedure one can show that for any $k\in
\{1,\dots,m-1\}$ there exists $C_k>0$ and $r_k>0$ such that for
any $r\ge r_k$
\begin{equation} \label{eq:proof-empty}
\Delta^{m-k}u(r)\ge
\begin{cases}
C_k r^{2k-1} & \qquad \text{if } n=1 \\[8pt]
C_k r^{2k-2} \log r & \qquad \text{if } n=2 \, .
\end{cases}
\end{equation}
After two further integrations, from \eqref{eq:proof-empty} with
$k=m-1$ we finally deduce that there exist $C_m>0$ and $r_m>0$
such that for any $r\ge r_m$
\begin{equation*} 
u(r)\ge
\begin{cases}
C_m r^{2m-1} & \qquad \text{if } n=1 \\[8pt]
C_m r^{2m-2} \log r & \qquad \text{if } n=2 \, .
\end{cases}
\end{equation*}
In particular we deduce that $\lim_{r\to +\infty} u(r)=+\infty$
and hence $u$ is bounded from below. We reached a contradiction
with Proposition \ref{noboundbelow}. This completes the proof of
the lemma.
\end{proof}

Suppose now that $n\ge 3$. We prove that if $\beta_{m-1}\ge 0$
then any solution $u_{\alpha,\beta}$ of \eqref{Cauchy} blows up in
finite time.

\begin{lemma} \label{l:not-global}
Let $n\ge 3$ and $m$ even. Then for any $\alpha\in\R$ and any
$\beta=(\beta_1,\dots,\beta_{m-1})\in\R^{m-1}$ with
$\beta_{m-1}\ge 0$, the corresponding solution $u_{\alpha,\beta}$
of \eqref{Cauchy} is not global.
\end{lemma}

\begin{proof} 
Let us denote the function $u_{\alpha,\beta}$ simply by $u$ and by
$[0,R)$ the corresponding maximal interval of continuation.
Suppose by contradiction that $R=+\infty$. As in the proof of
Lemma \ref{l:empty} we have that the function $r\mapsto r^{n-1}
(\Delta^{m-1}u)'(r)$ is increasing in $[0,+\infty)$ and being zero
at $r=0$ then $(\Delta^{m-1}u)'(r)>0$ for any $r>0$. Hence also
the map $r\mapsto \Delta^{m-1}u(r)$ is increasing in $[0,+\infty)$
and being $\Delta^{m-1}u(0)=\beta_{m-1}\ge 0$ then there exists
$C>0$ such that $\Delta^{m-1}u(r)\ge C$ for any $r\ge 1$.

After an iterative procedure as in Lemma \ref{l:empty}, one can
show that for any $k\in \{1,\dots,m-1\}$ there exists $C_k>0$ and
$r_k>0$ such that
\begin{equation} \label{eq:proof-empty-bis}
\Delta^{m-k}u(r)\ge C_k r^{2k-2} \qquad \text{for any } r\ge r_k
\, .
\end{equation}
After two further integrations in \eqref{eq:proof-empty-bis} with
$k=m-1$ we infer
\begin{equation*}
u(r)\ge C_m r^{2m-2} \qquad \text{for any } r\ge r_m \, ,
\end{equation*}
for some $C_m,r_m>0$. This shows that $u$ is bounded from below in
$[0,+\infty)$ in contradiction with Proposition
\ref{noboundbelow}. This completes the proof of the lemma.
\end{proof}

It is possible to provide a more detailed characterization of
blowing-up solutions of \eqref{Cauchy} as shown in the following
lemma.

\begin{lemma} \label{l:blow-up-bis}
Let $n\ge 3$ and $m$ even. For any $\alpha\in\R$ and
$\beta\in\R^{m-1}$ let $u_{\alpha,\beta}$ be the corresponding
solution of \eqref{Cauchy} with maximal interval of continuation
$[0,R_{\alpha,\beta})$, $R_{\alpha,\beta}\in (0,+\infty]$. Then
$R_{\alpha,\beta}<+\infty$ if and only if there exists $R_0\in
(0,R_{\alpha,\beta})$ such that
$\Delta^{m-1}u_{\alpha,\beta}(R_0)\ge 0$. Moreover in such a case
we also have
\begin{align} \label{eq:blow-up-0}
& \lim_{r\to R_{\alpha,\beta}^-} u_{\alpha,\beta}(r)=+\infty \, ,
\quad \lim_{r\to R_{\alpha,\beta}^-} u_{\alpha,\beta}'(r)=+\infty
\, ,  \\
\notag & \lim_{r\to R_{\alpha,\beta}^-} \Delta^k
u_{\alpha,\beta}(r)=+\infty \, , \quad \lim_{r\to
R_{\alpha,\beta}^-} (\Delta^k u_{\alpha,\beta})'(r)=+\infty \, ,
\end{align}
for any $k\in \{1,\dots,m-1\}$.
\end{lemma}

\begin{proof} For simplicity we write $u=u_{\alpha,\beta}$ and
$R=R_{\alpha,\beta}$. First suppose that $R<+\infty$. By
\eqref{Cauchy} we observe that $\Delta^{m-1} u$ is increasing and
hence admits a limit as $r\to R^{-}$. We claim that this limit is
$+\infty$. Suppose by contradiction that this limit is finite so
that $\Delta^{m-1}u$ is bounded in $[0,R)$. Successive
integrations imply that for any $k\in \{1,\dots,m-1\}$, $u$, $u'$,
$\Delta^k u$, $(\Delta^k u)'$ are bounded in $[0,R)$ and hence
also $u$ and all its derivatives are bounded in the same interval.
A standard argument in the theory of ordinary differential
equations leads to a contradiction with the maximality of $R$.
This completes the proof of the claim.

Since $\lim_{r\to R^-} \Delta^{m-1}u(r)=+\infty$, in particular
$\Delta^{m-1} u(r)>0$ for any $r$ in a sufficiently small left
neighborhood of $R$. After two integrations we deduce that
$(\Delta^{m-2}u)'$ and $\Delta^{m-2}u$ are bounded from below and
they admit a limit as $r\to R^-$. As above one shows that these
limits are necessarily $+\infty$.

Proceeding iteratively it is possible to prove that
\begin{align*} 
\lim_{r\to R^-} u(r)=+\infty \, , \quad \lim_{r\to R^-}
u'(r)=+\infty \, , \quad \lim_{r\to R^-} \Delta^k u(r)=+\infty \,
, \quad \lim_{r\to R^-} (\Delta^k u)'(r)=+\infty
\end{align*}
for any $k\in \{1,\dots,m-1\}$. This implies \eqref{eq:blow-up-0}
and in particular the existence of $R_0\in (0,R)$ such that
$\Delta^{m-1}u(R_0)\ge 0$. This completes the first part of the
proof.

Suppose now that there exists $R_0\in (0,R)$ such that
$\Delta^{m-1}u(R_0)\ge 0$. Proceeding by contradiction as in the
proof of Lemma \ref{l:not-global} we arrive to the conclusion.
\end{proof}

The next lemma is devoted to the behavior at infinity of global
solutions of \eqref{Cauchy}.

\begin{lemma} \label{l:limiti-Laplaciani}
Let $n\ge 3$ and $m$ even. Let $u$ be a global solution of
\eqref{Cauchy}. Then the following limits exist
\begin{equation} \label{eq:ex-lim-Lap}
\lim_{r\to +\infty} u(r) \, , \qquad \lim_{r\to +\infty}
r^{n-1}u'(r) \, \qquad \lim_{r\to +\infty} \Delta^k u(r) \, ,
\qquad \lim_{r\to +\infty} r^{n-1}(\Delta^k u)'(r) \, ,
\end{equation}
for any $k\in \{1,\dots,m-1\}$. Moreover
\begin{equation} \label{eq:neg-lim-Lap}
\lim_{r\to +\infty} u(r)=-\infty \, , \qquad \lim_{r\to +\infty}
\Delta^k u(r)\le 0
\end{equation}
for any $k\in \{1,\dots,m-1\}$.
\end{lemma}

\begin{proof} From \eqref{Cauchy} we deduce that the map $r\mapsto
r^{n-1}(\Delta^{m-1})'(r)$ is increasing and positive in
$(0,+\infty)$ and hence it admits a limit as $r\to +\infty$.
Moreover being $(\Delta^{m-1} u)'$ positive the function
$\Delta^{m-1}u$ is increasing in $(0,+\infty)$; hence it admits a
limit as $r\to +\infty$ and it is eventually of constant sign.

We can start again the procedure: the map $r\mapsto
r^{n-1}(\Delta^{m-2})'(r)$ is eventually monotone and hence it
admits a limit as $r\to +\infty$ and it is eventually of constant
sign. Therefore $\Delta^{m-2} u$ is eventually monotone; hence it
admits a limit as $r\to +\infty$ and it is eventually of constant
sign. An iteration of this procedure yields the validity of
\eqref{eq:ex-lim-Lap}.

It remains to prove \eqref{eq:neg-lim-Lap}. By
\eqref{eq:ex-lim-Lap} and Proposition \ref{noboundbelow} we
immediately have that $\lim_{r\to +\infty} u(r)=-\infty$.

Let us consider the second limit in \eqref{eq:neg-lim-Lap}.
Suppose by contradiction that there exists $k\in \{1,\dots,m-1\}$
such that $\lim_{r\to +\infty} \Delta^k u(r)>0$. Hence there exist
$C,\overline r>0$ such that
$$
\Delta^k u(r)>C \qquad \text{for any } r>\overline r
$$
After two integrations we obtain $\lim_{r\to +\infty} \Delta^{k-1}
u(r)=+\infty$ and hence $\lim_{r\to +\infty} u(r)=+\infty$, a
contradiction.
\end{proof}

The next two lemmas are devoted to a detailed description of the
set $\mathcal A_\alpha$ when $n\ge 3$.

\begin{lemma} \label{l:closed}
Let $n\ge 3$ and $m$ even. Then for any $\alpha\in\R$ the set
$\mathcal A_\alpha$ is closed.
\end{lemma}

\begin{proof} By Lemma \ref{l:not-global} we know that
$\R^{m-1}\setminus \mathcal A_\alpha\neq \emptyset$. We shall
prove that it is also open. Let $\beta_0\in\R^{m-1}\setminus
\mathcal A_\alpha$. By Lemma \ref{l:blow-up-bis} we may find
$R_0>0$ such that
\begin{equation*}
u_{\alpha,\beta_0}(R_0)>0 \, ,\quad u'_{\alpha,\beta_0}(R_0)>0 \,
,\quad \Delta^k u_{\alpha,\beta_0}(R_0)>0 \, ,\quad (\Delta^k
u_{\alpha,\beta_0})'(R_0)>0 \, ,
\end{equation*}
for any $k\in \{1,\dots,m-1\}$. By Proposition
\ref{dipendenza-continua} we deduce that there exists $\delta>0$
such that for any $\beta\in B(\beta_0,\delta)$ the function
$u_{\alpha,\beta}$ is well-defined at $R_0$ and moreover
\begin{equation*}
u_{\alpha,\beta}(R_0)>0 \, ,\quad u'_{\alpha,\beta}(R_0)>0 \,
,\quad \Delta^k u_{\alpha,\beta}(R_0)>0 \, ,\quad (\Delta^k
u_{\alpha,\beta})'(R_0)>0 \, .
\end{equation*}
Here we denoted by $B(\beta_0,\delta)$ the open ball in $\R^{m-1}$
of radius $\delta$ centered at $\beta_0$. Applying Lemma
\ref{l:blow-up-bis} to these functions $u_{\alpha,\beta}$ we infer
that they are not global thus showing that
$B(\beta_0,\delta)\subseteq \R^{m-1}\setminus \mathcal A_\alpha$.
This completes the proof of the lemma.
\end{proof}

We now prove that for any $\alpha\in\R$ the set $\mathcal
A_\alpha$ is not empty in dimension $n\ge 3$.

\begin{lemma} \label{l:non-empty}
Let $n\ge 3$ and $m$ even. Then the following statements hold:
\begin{itemize}
\item[(i)] for any $\alpha\in\R$ the set $\mathcal A_\alpha$ is
nonempty;

\item[(ii)] for any $\alpha\in\R$ there exists a function
$\Phi_\alpha:\R^{m-2}\to (-\infty,0)$ such that
$$
\mathcal A_\alpha=\{\beta=(\beta_1,\dots,\beta_{m-1})\in
\R^{m-1}:\beta_{m-1}\le \Phi(\beta_1,\dots,\beta_{m-2}) \} \, ;
$$
\item[(iii)] for any $\alpha\in\R$, $\Phi_\alpha$ is a continuous
function, $\partial \mathcal A_\alpha$ coincides with the graph of
$\Phi_\alpha$ and
$$
\ds{\mathop{\mathcal
A_\alpha}^{\circ}}=\{\beta=(\beta_1,\dots,\beta_{m-1})\in
\R^{m-1}:\beta_{m-1}< \Phi(\beta_1,\dots,\beta_{m-2}) \} \, .
$$
\end{itemize}
\end{lemma}

\begin{proof}  (i)-(ii) Let $\beta_1,\dots,\beta_{m-2}\in \R$ be
fixed arbitrarily. Put $\beta_{m-1}=b$ where $b$ is a parameter
varying in $(-\infty,0)$ and define $u_b$ as the unique solution
of \eqref{Cauchy} corresponding to the initial values
$\alpha_0=\alpha$, $\alpha_k=\beta_k$ for any $k\in
\{1,\dots,m-2\}$ and $\alpha_{m-1}=b$. Denote by $(0,R_b)$ with
$R_b\in (0,+\infty]$ the maximal interval of continuation of the
solution $u_b$. We shall prove that for any $b<0$ small enough
then $R_b=+\infty$. For any $b<0$ let
$$
M_b:=\sup\left\{r\in (0,R_b): \Delta^{m-1} u_b(s)<\frac b2 \  \
\text{for any } s\in [0,r)\right\} \, .
$$
We claim that there exits $b<0$ such that $M_b=R_b$. We first show
that from this claim we easily arrive to the conclusion of the
proof. Indeed if $b<0$ is such that $M_b=R_b$ then $\Delta^{m-1}
u_b(r)<\frac b2$ for any $r\in [0,R_b)$. If $R_b$ were finite then
after successive integrations one easily shows that $u_b$ is
bounded from above in $[0,R_b)$ and, in turn, that $e^{u_b}$ is
bounded in $[0,R_b)$. Therefore by \eqref{Cauchy} and successive
integrations one can show that $u_b$ and all its derivatives until
order $2m-1$ are bounded in $[0,R_b)$. A standard argument in the
theory of ordinary differential equations leads to a contradiction
with the maximality of $R_b$. Therefore, by Proposition
\ref{mckenna reichel}, Lemma \ref{l:not-global} and Lemma
\ref{l:closed} we infer that there exists $b_0<0$ such that
$$
\{b\in\R:u_b \ \text{is a global solution of
\eqref{Cauchy}}\}=(-\infty,b_0] \, .
$$
Finally it is sufficient to put
$\Phi_\alpha(\beta_1,\dots,\beta_{m-2}):=b_0$.

Let us prove that claim. We proceed by contradiction assuming that
$M_b<R_b$ for any $b<0$. By definition of $M_b$ we have that
$\Delta^{m-1}u_b(r)\le \frac b2$ for any $r\in [0,M_b]$ and that
$\Delta^{m-1}u_b(M_b)=\frac b2$. In the rest of the proof we use
the notation $\sum_{j=k_1}^{k_2} a_j=0$ and $\prod_{j=k_1}^{k_2}
a_j=1$ whenever $k_1>k_2$.

Then, for any $k\in \{2,\dots,m-1\}$ and $r\in [0,M_b]$ we have
\begin{equation*} 
\Delta^{m-k}u_b(r)\le \frac{b}{2^{k} \ (k-1)! \
\prod_{l=1}^{k-1}(n+2l-2)} \, r^{2k-2} +\sum_{j=0}^{k-2}
\frac{\beta_{m-k+j}}{2^j \ j! \ \prod_{l=1}^{j}(n+2l-2)} \, r^{2j}
 \, ,
\end{equation*}
\begin{equation*} 
(\Delta^{m-k}u_b)'(r)\le \frac{b}{2^{k-1} \ (k-2)! \
\prod_{l=1}^{k-1}(n+2l-2)} \, r^{2k-3}
  +\sum_{j=1}^{k-2} \frac{\beta_{m-k+j} \
r^{2j-1}}{2^{j-1} \ (j-1)! \ \prod_{l=1}^{j}(n+2l-2)} \, .
\end{equation*}
Finally other two integrations yield
$$
u'_b(r)\le \frac{b}{2^{m-1} \ (m-2)! \ \prod_{l=1}^{m-1}(n+2l-2)}
\, r^{2m-3}
  +\sum_{j=1}^{m-2} \frac{\beta_{j} \
r^{2j-1}}{2^{j-1} \ (j-1)! \ \prod_{l=1}^{j}(n+2l-2)} \,  ,
$$
and
$$
u_b(r)\le \alpha+\frac{b}{2^{m} \ (m-1)! \
\prod_{l=1}^{m-1}(n+2l-2)} \, r^{2m-2}
  +\sum_{j=1}^{m-2} \frac{\beta_{j} \
r^{2j}}{2^{j} \ j! \ \prod_{l=1}^{j}(n+2l-2)}=:P_b(r) \, ,
$$
for any $r\in [0,M_b]$. The function $P_b$ is a polynomial of
degree $2m-2$ which admits the representation
$$
P_b(r)=C_{n,m}br^{2m-2}+Q(r)
$$
where $C_{n,m}=\left[2^{m} \ (m-1)! \
\prod_{l=1}^{m-1}(n+2l-2)\right]^{-1}$ and $Q$ is a polynomial of
degree $2m-4$. For any $b<-1$, by \eqref{Cauchy}, we then obtain
\begin{align*}
\Delta^{m-1}u_b(r)&=b+\int_0^r s^{1-n} \left(\int_0^s
t^{n-1}e^{u_b(t)} \, dt\right)ds\le b+\int_0^r s^{1-n}
\left(\int_0^s t^{n-1}e^{P_b(t)} \, dt\right)ds \\
& \le b+\int_0^r s^{1-n} \left(\int_0^s t^{n-1}e^{P_{-1}(t)} \,
dt\right)ds  \qquad \text{for any }  r\in [0,M_b] \, .
\end{align*}
We remark that since $n\ge 3$ and $b$ is negative then the
function $s\mapsto s^{1-n} \left(\int_0^s t^{n-1}e^{P_{-1}(t)} \,
dt\right)$ is integrable in $(0,+\infty)$ so that we may write
\begin{align*}
\Delta^{m-1}u_b(r) \le b+\int_0^{\infty} s^{1-n} \left(\int_0^s
t^{n-1}e^{P_{-1}(t)} \, dt\right)ds  \qquad \text{for any }  r\in
[0,M_b] \, .
\end{align*}
In particular for $r=M_b$ we obtain
$$
\frac b2\le b+\int_0^{\infty} s^{1-n} \left(\int_0^s
t^{n-1}e^{P_{-1}(t)} \, dt\right)ds \qquad \text{for any } b<-1 \,
,
$$
and a contradiction follows by letting $b\to -\infty$.

(iii) Let $\beta'_0=(\beta_{0,1},\dots,\beta_{0,m-2})$ a point in
$\R^{m-2}$ and let $\beta_0=(\beta'_0,\Phi_\alpha(\beta'_0))$. We
shall prove that $\Phi_\alpha$ is continuous in $\beta'_0$. We
observe that by Lemma \ref{l:closed} the subgraph of $\Phi_\alpha$
is closed and hence $\Phi_\alpha$ is upper semicontinuous. It
remains to prove that $\Phi_\alpha$ is also lower semicontinuous.
In the rest of the proof we denote by $\beta\in\R^{m-1}$ the point
$\beta:=(\beta',\Phi_\alpha(\beta'_0)-\eps)$ and by
$|\cdot|_{\infty}$ the norm
$$
|\gamma|_\infty:=\max_{1\le k\le m-2} |\gamma_k| \qquad \text{for
any } \gamma\in \R^{m-2} \, .
$$
For $0<\eta<\eps$ and $\beta'\in\R^{m-2}$ let us define
$$
M_{\eta,\beta'}:=\sup\left\{r>0:\Delta^{m-1}u_{\alpha,\beta}(s)\le
\Delta^{m-1}u_{\alpha,\beta_0}(s)-\eta
 \ \text{for any } s\in [0,r]\right\}\in (0,R_{\alpha,\beta}] \, .
$$
We divide the proof of (iii) into three steps.

{\bf Step 1.} We claim that for any $0<\eta<\eps$ there exist
$\overline\delta>0$ and $K\in (0,R_{\alpha,\beta})$ such that if
$\delta\in(0,\overline\delta)$
\begin{equation} \label{claim-0}
 |\beta'-\beta'_0|_\infty<\delta \quad \text{and} \quad
M_{\eta,\beta'}<R_{\alpha,\beta} \quad \Longrightarrow \quad
M_{\eta,\beta'}\le K  \, .
\end{equation}
Proceeding by contradiction we would find $0<\eta<\eps$ such that
for any $\overline\delta>0$ and $K\in (0,R_{\alpha,\beta})$, there
exist $0<\delta<\overline\delta$ and $\beta'\in\R^{m-2}$ such that
$|\beta'-\beta'_0|_\infty<\delta$ and
$K<M_{\eta,\beta'}<R_{\alpha,\beta}$.

Let us put $U(r)=u_{\alpha,\beta}(r)-u_{\alpha,\beta_0}(r)$ for
any $r\in [0,R_{\alpha,\beta})$. Proceeding as in the proof of
(i)-(ii) we obtain for any $r\in [0,M_{\eta,\beta'}]$ and $k\in
\{1,\dots,m-1\}$
\begin{equation} \label{eq:Laplaciani-U}
\Delta^{m-k}U(r)\le -\frac{\eta r^{2k-2}}{2^{k-1} \ (k-1)! \
\prod_{l=1}^{k-1}(n+2l-2)}+\sum_{j=0}^{k-2} \frac{\delta
r^{2j}}{2^j \ j! \ \prod_{l=1}^{j}(n+2l-2)}=:P_{\eta,\delta,k}(r)
 \, ,
\end{equation}
\begin{equation} \label{eq:derivata-Laplaciani-U}
(\Delta^{m-k}U)'(r)\le -\frac{\eta r^{2k-3}}{2^{k-2} \ (k-2)! \
\prod_{l=1}^{k-1}(n+2l-2)}
  +\sum_{j=1}^{k-2} \frac{\delta
r^{2j-1}}{2^{j-1} \ (j-1)! \
\prod_{l=1}^{j}(n+2l-2)}=:Q_{\eta,\delta,k}(r) \, ,
\end{equation}
\begin{equation} \label{eq:derivata-U}
U'(r)\le -\frac{\eta r^{2m-3}}{2^{m-2} \ (m-2)! \
\prod_{l=1}^{m-1}(n+2l-2)}
  +\sum_{j=1}^{m-2} \frac{\delta  r^{2j-1}
}{2^{j-1} \ (j-1)! \
\prod_{l=1}^{j}(n+2l-2)}=:Q_{\eta,\delta,m}(r)
\end{equation}
and
\begin{equation} \label{eq:funzione-U}
U(r)\le -\frac{\eta r^{2m-2}}{2^{m-1} \ (m-1)! \
\prod_{l=1}^{m-1}(n+2l-2)}
  +\sum_{j=1}^{m-2} \frac{\delta r^{2j}
}{2^{j} \ j! \ \prod_{l=1}^{j}(n+2l-2)}=:P_{\eta,\delta,m}(r) \, .
\end{equation}
We may choose $K$ and $\overline\delta$ such that
$$
P_{\eta,\overline\delta,k}(r)<0 \, , \quad
Q_{\eta,\overline\delta,k}(r)<0 \qquad \text{for any } r\ge K
\text{ and } k\in \{1,\dots,m\} \, .
$$
In particular by \eqref{eq:Laplaciani-U}-\eqref{eq:funzione-U}
with $r=M_{\eta,\beta'}$ we infer
\begin{equation*}
U(M_{\eta,\beta'})<0\, , \quad U'(M_{\eta,\beta'})<0\, , \quad
\Delta^j U(M_{\eta,\beta'})<0\, , \quad (\Delta^j
U)'(M_{\eta,\beta'})<0 \, ,
\end{equation*}
for any $j\in \{1,\dots,m-1\}$. Therefore by Proposition
\ref{mckenna reichel} we obtain
\begin{equation*}
u_{\alpha,\beta}(r)\le u_{\alpha,\beta_0}(r) \, , \ \
u'_{\alpha,\beta}(r)\le u'_{\alpha,\beta_0}(r) \, , \ \ \Delta^j
u_{\alpha,\beta}(r)\le \Delta^j u_{\alpha,\beta_0}(r) \, , \ \
(\Delta^j u_{\alpha,\beta})'(r)\le (\Delta^j
u_{\alpha,\beta_0})'(r) \, ,
\end{equation*}
for any $r\in [M_{\eta,\beta'},R_{\alpha,\beta})$ and $j\in
\{1,\dots,m-1\}$. In particular for any $r\in
(M_{\eta,\beta'},R_{\alpha,\beta})$, we obtain
\begin{align*}
\Delta^{m-1} u_{\alpha,\beta}(r)&=\Delta^{m-1}
u_{\alpha,\beta}(M_{\eta,\beta'})+\int_{M_{\eta,\beta'}}^r
(\Delta^{m-1} u_{\alpha,\beta})'(s) \, ds \\
& \le \Delta^{m-1}
u_{\alpha,\beta_0}(M_{\eta,\beta'})-\eta+\int_{M_{\eta,\beta'}}^r
(\Delta^{m-1} u_{\alpha,\beta_0})'(s) \, ds= \Delta^{m-1}
u_{\alpha,\beta_0}(r)-\eta \, ,
\end{align*}
contradicting the maximality of $M_{\eta,\beta'}$. This proves
\eqref{claim-0}.

{\bf Step 2.} We claim that there exist $0<\eta<\eps$ and
$\delta>0$ such that for any $\beta'\in \R^{m-2}$ with
$|\beta'-\beta'_0|_\infty<\delta$, we have
$M_{\eta,\beta'}=R_{\alpha,\beta}$. Suppose by contradiction that
for any $0<\eta<\eps$ and for any $\delta>0$ there exists
$\beta'\in \R^{m-2}$ such that $|\beta'-\beta'_0|_\infty<\delta$
and $M_{\eta,\beta'}<R_{\alpha,\beta}$.

Let $\overline \delta$ and $K$ be as in Step 1. By Proposition
\ref{dipendenza-continua}, up to shrinking $\overline\delta$ if
necessary, we have that $u_{\alpha,\beta}$ is well defined in
$[0,K]$ for any $\beta'$ satisfying
$|\beta'-\beta'_0|_\infty<\delta<\overline\delta$. Moreover
$u_{\alpha,\beta}$ converges uniformly in $[0,K]$ to the function
$u_{\alpha,(\beta_0',\Phi_\alpha(\beta_0')-\eps)}$ as $\beta'\to
\beta'_0$. Hence by Proposition \ref{mckenna reichel} we have that
for any $\sigma>0$ we may shrink $\overline\delta$ in such a way
that
\begin{equation} \label{eq:**}
u_{\alpha,\beta}(r)<u_{\alpha,(\beta_0',\Phi_\alpha(\beta_0')-\eps)}(r)+\sigma
\le u_{\alpha,\beta_0}(r)+\sigma \qquad \text{for any } r\in [0,K]
\end{equation}
with $\beta$ such that
$|\beta'-\beta'_0|_\infty<\delta<\overline\delta$ and
$M_{\eta,\beta'}<R_{\alpha,\beta}$.

By \eqref{Cauchy}, \eqref{claim-0} and \eqref{eq:**}, we obtain
\begin{equation} \label{eq:??}
\Delta^{m-1}u_{\alpha,\beta}(r)-\Delta^{m-1}u_{\alpha,\beta}(0)
\le e^\sigma \Delta^{m-1}u_{\alpha,\beta_0}(r)-e^\sigma
\Delta^{m-1}u_{\alpha,\beta_0}(0)
 \qquad \text{for any } r\in [0,M_{\eta,\beta'}] \, .
\end{equation}
Substituting $r=M_{\eta,\beta'}$ in \eqref{eq:??} and taking into
account that
$$
\Delta^{m-1}u_{\alpha,\beta}(0)=\Delta^{m-1}u_{\alpha,\beta_0}(0)-\varepsilon\,
, \quad
\Delta^{m-1}u_{\alpha,\beta}(M_{\eta,\beta'})=\Delta^{m-1}u_{\alpha,\beta_0}(M_{\eta,\beta'})-\eta
\quad \text{and} \quad M_{\eta,\beta'}\le K \, ,
$$
we obtain
$$
\Delta^{m-1} u_{\alpha,\beta_0}(M_{\eta,\beta'})-\eta\le e^\sigma
\Delta^{m-1}u_{\alpha,\beta_0}(M_{\eta,\beta'})+(1-e^\sigma)
\Delta^{m-1} u_{\alpha,\beta_0}(0)-\eps
$$
for any $\eta\in (0,\eps)$ and $\sigma>0$. Letting $\sigma\to 0^+$
and then $\eta\to 0^+$ we reach a contradiction. This completes
the proof of Step 2.

{\bf Step 3.} In this step we complete the proof of (iii). By Step
2 we have that for any $|\beta'-\beta'_0|_\infty<\delta$
$$
\Delta^{m-1} u_{\alpha,\beta}(r)\le \Delta^{m-1}
u_{\alpha,\beta_0}(r)-\eta<0 \qquad \text{for any } r\in
[0,R_{\alpha,\beta}) \, ,
$$
where the last inequality follows from Lemma \ref{l:blow-up-bis}.
By Lemma \ref{l:blow-up-bis} we also deduce that
$u_{\alpha,\beta}$ is a global solution of \eqref{Cauchy}. By
(i)-(ii), this implies that $\Phi_\alpha(\beta')\ge
\Phi_{\alpha}(\beta'_0)-\eps$ for any $\beta'$ satisfying
$|\beta'-\beta'_0|_\infty<\delta$. Hence
$\Phi_{\alpha}(\beta'_0)\le \liminf_{\beta'\to\beta'_0}
\Phi_\alpha(\beta')$ which together with the upper semicontinuity
gives the continuity of $\Phi_\alpha$ at $\beta_0'$. Since
$\Phi_\alpha$ is continuous then the set
$$
\{\beta=(\beta',\beta_{m-1})\in\R^{m-1}:\beta_{m-1}<\Phi_\alpha(\beta')\}
$$
is open and hence the proof of (iii) follows.
\end{proof}

In order to better understand the asymptotic behavior of global
solutions of \eqref{Cauchy} and the behavior of the function
$\Phi_\alpha$ introduced in Lemma \ref{l:non-empty}, we prove some
auxiliary results.

\begin{lemma} Let $n\ge 3$ and $m$ even. Consider the equation
\begin{equation} \label{eq:Delta^m}
\Delta^{m} U(r)=\frac{1}{r^3} \qquad \text{for any } r>0 \, .
\end{equation}
Then \eqref{eq:Delta^m} admits a solution in the form
\begin{equation} \label{eq:def-U}
U(r)=
\begin{cases}
C_{n,m}r^{2m-3}+\log\lambda_{n,m} & \qquad \text{if } n\ge 4
\\[8pt]
C_{n,m}r^{2m-3}(\log r+D_{n,m})+\log\lambda_{n,m} & \qquad
\text{if } n=3
\end{cases}
\end{equation}
where $C_{n,m}$ is the negative constant defined by
$$
C_{n,m}:=
\begin{cases} \left[\prod_{j=1}^m (2j-3)\cdot
\prod_{j=0}^{m-1} (n+2j-3)
\right]^{-1} & \qquad \text{if } n\ge 4 \\[8pt]
\left[\prod_{j=1}^m (2j-3)\cdot \prod_{j=1}^{m-1} 2j \right]^{-1}
& \qquad \text{if } n=3 \, ,
\end{cases}
$$
$$
\lambda_{n,m}=
\begin{cases}
{\ds\min_{r\in (0,+\infty)}} \frac{\exp[|C_{n,m}|r^{2m-3}]}{r^3}
& \qquad \text{if } n\ge 4 \\[8pt]
{\ds\min_{r\in (0,+\infty)}} \frac{\exp[|C_{n,m}|r^{2m-3}(\log
r+D_{n,m})]}{r^3} & \qquad \text{if } n=3
\end{cases}
$$
and $D_{n,m}\in \R$ is a suitable constant.

Moreover $U$ satisfies
\begin{equation} \label{eq:Delta^m U=e^U}
\Delta^m U(r)\ge e^{U(r)} \qquad \text{for any } r>0 \, .
\end{equation}
\end{lemma}

\begin{proof} We proceed in this way: let $U=U(r)$ a function
satisfying \eqref{eq:Delta^m}. If $n\ge 4$, after an iterative
procedure of integration we may assume that $U$ satisfies
\begin{equation*}
\Delta^{m-k} U(r)=\left[\prod_{j=1}^k (2j-3)\cdot
\prod_{j=0}^{k-1} (n+2j-3) \right]^{-1} r^{2k-3} \qquad \text{for
any } r>0
\end{equation*}
and
\begin{equation*}
(\Delta^{m-k} U)'(r)=\left[\prod_{j=1}^{k-1} (2j-3)\cdot
\prod_{j=0}^{k-1} (n+2j-3) \right]^{-1} r^{2k-4} \qquad \text{for
any } r>0
\end{equation*}
for any $k\in \{1,\dots,m-1\}$ where we put $\prod_{j=1}^0
(2j-3)=1$. Taking $k=m-1$ in the previous identities and
integrating we also have
\begin{equation*}
U'(r)=\left[\prod_{j=1}^{m-1} (2j-3)\cdot \prod_{j=0}^{m-1}
(n+2j-3) \right]^{-1} r^{2m-4} \qquad \text{for any } r>0 \, .
\end{equation*}
Therefore we may choose $U$ as in \eqref{eq:def-U}. We proceed in
a similar way in the case $n=3$.

Finally the fact that $U$ solves \eqref{eq:Delta^m U=e^U} is a
consequence of the definition of $\lambda_{n,m}$.
\end{proof}

\begin{lemma} \label{l:asympt-1}
Let $n\ge 3$ and $m$ even. For any $\alpha\in\R$ the following
facts hold true:
\begin{itemize}
\item[(i)] if  $\beta\in\partial \mathcal A_\alpha$ then
$$
\lim_{r\to +\infty} \Delta^{m-1} u_{\alpha,\beta}(r)=0;
$$

\item[(ii)] if $\beta\in \ds{\mathop{\mathcal A_\alpha}^{\circ}}$
then
$$
\lim_{r\to +\infty} \Delta^{m-1} u_{\alpha,\beta}(r)=\ell \in
(-\infty,0) \, ,
$$
\begin{equation*}
\Delta^{m-k} u_{\alpha,\beta}(r)\sim \frac{\ell}{2^{k-1} \ (k-1)!
\ \prod_{l=1}^{k-1} (n+2l-2)} \, r^{2k-2} \qquad \text{as } r\to
+\infty
\end{equation*}
for any $k\in \{2,\dots,m-1\}$ and
\begin{equation} \label{eq:u-sim-r2m-2}
u_{\alpha,\beta}(r) \sim \frac{\ell}{2^{m-1} \ (m-1)! \
\prod_{l=1}^{m-1} (n+2l-2)} \, r^{2m-2} \qquad \text{as } r\to
+\infty \, .
\end{equation}
\end{itemize}
\end{lemma}

\begin{proof} (i) Suppose by contradiction that $\ell:=\lim_{r\to
+\infty} \Delta^{m-1} u_{\alpha,\beta}(r)<0$. We recall that the
case $\ell>0$ can be excluded immediately thanks to
\eqref{eq:neg-lim-Lap}. We claim that $\ell$ is finite. Suppose by
contradiction that $\ell=-\infty$. Then after an iterative
procedure of integration we find that for any $M>0$ there exists
$\overline r>0$ such that
$$
u_{\alpha,\beta}(r)<-M r^{2m-2} \qquad \text{for any } r>\overline
r
$$
so that the map $r\mapsto r^{n-1}e^{u_{\alpha,\beta}(r)}\in
L^1(0,+\infty)$.

Hence by \eqref{Cauchy} we have
$(\Delta^{m-1}u_{\alpha,\beta})'(r)=r^{1-n}\int_0^r
s^{n-1}e^{u_{\alpha,\beta}(s)} ds\in L^1(0,+\infty)$ since $n\ge
3$, in contradiction with $\ell=-\infty$. From now on we may
assume that $\ell\in (-\infty,0)$.

Then, since $n\ge 3$, after integration one obtains
\begin{equation} \label{eq:sim-3}
(\Delta^{m-k} u_{\alpha,\beta})'(r)\sim \ell
\left[\prod_{j=1}^{k-2} 2j \cdot \prod_{j=1}^{k-1}
(n+2j-2)\right]^{-1} r^{2k-3}  \qquad \text{as } r\to +\infty
\end{equation}
and
\begin{equation} \label{eq:sim-4}
\Delta^{m-k} u_{\alpha,\beta}(r)\sim \ell \left[\prod_{j=1}^{k-1}
2j \cdot \prod_{j=1}^{k-1} (n+2j-2)\right]^{-1} r^{2k-2} \qquad
\text{as } r\to +\infty
\end{equation}
for any $k\in \{2,\dots,m-1\}$ where we put $\prod_{j=1}^0 2j=1$.
Moreover we also have
\begin{equation} \label{eq:sim-5}
u_{\alpha,\beta}'(r)\sim \ell \left[\prod_{j=1}^{m-2} 2j \cdot
\prod_{j=1}^{m-1} (n+2j-2)\right]^{-1} r^{2m-3}  \qquad \text{as }
r\to +\infty
\end{equation}
and
\begin{equation} \label{eq:sim-6}
u_{\alpha,\beta}(r)\sim \ell \left[\prod_{j=1}^{m-1} 2j \cdot
\prod_{j=1}^{m-1} (n+2j-2)\right]^{-1} r^{2m-2} \qquad \text{as }
r\to +\infty \, .
\end{equation}
Combining \eqref{eq:sim-3}-\eqref{eq:sim-6} with \eqref{eq:def-U}
we infer that there exists $\overline r$ such that
\begin{equation} \label{eq:ine-overline}
u_{\alpha,\beta}(\overline r)<U(\overline r)\, , \quad
u_{\alpha,\beta}'(\overline r)<U'(\overline r) \, ,\quad \Delta^k
u_{\alpha,\beta}(\overline r)<\Delta^k U(\overline r) \, , \quad
(\Delta^k u_{\alpha,\beta})'(\overline r)<(\Delta^k U)'(\overline
r)
\end{equation}
for any $k\in \{1,\dots,m-1\}$. By \eqref{eq:Delta^m U=e^U} and
Proposition \ref{mckenna reichel} we deduce that the above
inequalities hold not only at $\overline r$ but at any
$r>\overline r$. Then if we write $\beta$ in the form
$(\beta',\beta_{m-1})$ with $\beta'\in \R^{m-2}$ and
$\beta_{m-1}=\Phi_\alpha(\beta')$, and if we define
$\widetilde\beta:=(\beta',\gamma)$ with $\gamma>\beta_{m-1}$
sufficiently close to $\beta_{m-1}$, we deduce that
\eqref{eq:ine-overline} also holds with $u_{\alpha,\gamma}$ in
place of $u_{\alpha,\beta}$. Exploiting again \eqref{eq:Delta^m
U=e^U} and Proposition \ref{mckenna reichel} it follows that
$u_{\alpha,\gamma}$ is a global solution of \eqref{Cauchy} in
contradiction with
 the maximality of $\beta_{m-1}$.

(ii) Let us write $\beta$ in the form $(\beta',\beta_{m-1})$ with
$\beta'\in \R^{m-2}$ and define
$\beta_0:=(\beta',\Phi_\alpha(\beta'))$ so that
$\beta_{m-1}<\Phi_\alpha(\beta')$. Put
$v:=u_{\alpha,\beta}-u_{\alpha,\beta_0}$ so that by Proposition
\ref{mckenna reichel}, $\Delta^m v(r)\le 0$ for any $r>0$,
$\Delta^{m-1} v(0)=\beta_{m-1}-\Phi_\alpha(\beta')<0$ and
$\Delta^k v(0)=0$ for any $k\in \{1,\dots,m-2\}$. After
integration it follows that $\Delta^{m-1} v(r)\le
\beta_{m-1}-\Phi_\alpha(\beta')$ for any $r\ge 0$. Further
integrations then imply $\lim_{r\to +\infty} \Delta^k v(r)<0$ for
any $k\in \{1,\dots,m-1\}$. Hence, by \eqref{eq:neg-lim-Lap} we
deduce that
\begin{equation} \label{eq:lim-neg}
\lim_{r\to +\infty} \Delta^k u_{\alpha,\beta} (r)\le \lim_{r\to
+\infty} \Delta^k v(r)<0
\end{equation}
for any $k\in \{1,\dots,m-1\}$. In particular if we choose $k=1$
then we infer that $\Delta u_{\alpha,\beta}(r)<-C$ for some
constant $C>0$ for $r$ large enough. A couple of integrations then
yield
$$
u_{\alpha,\beta}(r)<-C' r^2 \qquad \text{for any } r>\overline r
$$
for some $C',\overline r>0$. Then, proceeding as in the proof of
(i), one can show that $\ell:=\lim_{r\to +\infty}
\Delta^{m-1}u_{\alpha,\beta}(r)$ is finite and moreover by
\eqref{eq:lim-neg} we have $\ell\in (-\infty,0)$.

After an iterative procedure of integration the proof of the
remaining part of (ii) follows.
\end{proof}

When $\beta\in \partial\mathcal A_\alpha$ estimate
\eqref{eq:u-sim-r2m-2} is no more true. However a suitable
estimate from above can be proved:

\begin{lemma} \label{l:slow-decay}
Let $n\ge 3$ and $m$ even. Let $\alpha\in\R$ and let
$\beta=(\beta_1,\dots,\beta_{m-1})=(\beta',\beta_{m-1})$ be such
that $\beta_{m-1}=\Phi_\alpha(\beta')$.
Then
$$
u_{\alpha,\beta}(r)=o(r^{2m-2}) \qquad \text{as } r\to +\infty
$$
and moreover
\begin{equation*}
u_{\alpha,\beta}(r)\le -2m\log r+O(1) \qquad \text{as } r\to
+\infty \, .
\end{equation*}
\end{lemma}

\begin{proof} The first assertion of the lemma is a consequence of Lemma \ref{l:asympt-1}
(i).

Let us prove the second assertion. If there exists $k\in
\{1,\dots,m-1\}$ such that $\lim_{r\to +\infty} \Delta^k
u_{\alpha,\beta}(r)<0$, after a finite number of integrations we
observe that $u_{\alpha,\beta}$ diverges to $-\infty$ as $r\to
+\infty$ with the rate of a positive power of $r$ and hence the
conclusion of the lemma trivially follows. For this reason thanks
to Lemma \ref{l:limiti-Laplaciani}, in the rest of the proof it is
not restrictive assuming that
\begin{equation} \label{eq:lim-null}
\lim_{r\to +\infty} \Delta^k u_{\alpha,\beta}(r)=0 \qquad \text{
for any } k\in \{1,\dots,m-1\} \, .
\end{equation}
We proceed similarly to the proof of Lemma 1 in \cite{fergru}.
Suppose by contradiction that $u_{\alpha,\beta}(r)+2m\log r$ is
not bounded from above and let $r_j\uparrow +\infty$ be such that
$M_j:=u_{\alpha,\beta}(r_j)+2m\log r_j\to +\infty$ as $j\to
+\infty$. Next we define $u_j(r)=u_{\alpha,\beta}(r_j r)+2m\log
r_j-M_j$ in such a way that $u_j$ vanishes on $\partial B_1$ and
it solves the equation $\Delta^m u_j=\lambda_j e^{u_j}$ in $B_1$
where we put $\lambda_j:=e^{M_j}$.

By \eqref{Cauchy}, \eqref{eq:lim-null} and successive
integrations, one may check that $(-1)^k \Delta^k
u_{\alpha,\beta}(r)>0$ for any $r>0$ and $k\in \{1,\dots,m-1\}$.

Resuming the above information we deduce that $u_j$ satisfies
\begin{equation*}
\begin{cases}
\Delta^m u_j=\lambda_j e^{u_j} & \quad \text{in } B_1 \\
u_j=0  & \quad \text{on } \partial B_1 \\
(-\Delta)^k u_j>0 & \quad \text{on } \partial B_1 \ \ \text{for
any } k\in\{1,\dots,m-1\} \, .
\end{cases}
\end{equation*}
This means that $u_j$ is a supersolution for the following Navier
boundary value problem
\begin{equation*}
\begin{cases}
\Delta^m u=\lambda_j e^{u} & \quad \text{in } B_1 \\
u=0  & \quad \text{on } \partial B_1 \\
\Delta u=\dots = \Delta^{m-1} u=0 & \quad \text{on } \partial B_1
\, .
\end{cases}
\end{equation*}
One may check that such a problem admits a solution also in a weak
sense only if $\lambda_j\le \lambda^*$ where $\lambda^*\in
(0,+\infty)$ is a suitable extremal value for the existence of a
solution, see \cite{berchiogazzola} for more details in the case
$m=2$. But $\lambda_j \to +\infty$ as $j\to +\infty$ thus
producing a contradiction.
\end{proof}

As a consequence of Lemma \ref{l:asympt-1} (i) we prove

\begin{lemma} \label{l:non-constant}
Let $n\ge 3$ and $m\ge 4$ even. Then for any $\alpha\in\R$,
$\Phi_\alpha$ is decreasing with respect to each variable. In
other words the map $t\mapsto
\Phi_\alpha(\beta_1,\dots,\beta_{k-1},t,\beta_{k+1},\dots,\beta_{m-2})$
is decreasing in $\R$ for any $k\in\{1,\dots,m-2\}$.
\end{lemma}

\begin{proof} 
Let $t<s$ and let $u_t$ and $u_s$ be the solutions of
\eqref{Cauchy} corresponding respectively to the initial values
$$
(\alpha,\beta_1,\dots,\beta_{k-1},t,\beta_{k+1},\dots,\beta_{m-2},\gamma_t)
\, , \qquad
(\alpha,\beta_1,\dots,\beta_{k-1},s,\beta_{k+1},\dots,\beta_{m-2},\gamma_s)
$$
where we put $\gamma_t:=
\Phi_\alpha(\beta_1,\dots,\beta_{k-1},t,\beta_{k+1},\dots,\beta_{m-2})$
and
$\gamma_s:=\Phi_\alpha(\beta_1,\dots,\beta_{k-1},s,\beta_{k+1},\dots,\beta_{m-2})$.

Suppose by contradiction that $\gamma_t\le \gamma_s$. Then by
Proposition \ref{mckenna reichel} we deduce that
\begin{align} \label{eq:inequalities}
& \Delta^k u_t(r)<\Delta^k u_s(r) \, , \quad (\Delta^k
u_t)'(r)<(\Delta^k u_s)'(r) \qquad \text{for any } r>0 \
\text{and} \ k\in \{1,\dots,m-1\} \, , \\
\notag & u_t(r)<u_s(r) \, , \quad  u_t'(r)<u_s'(r) \qquad
\text{for any } r>0 \, .
\end{align}
By \eqref{Cauchy} and Lemma \ref{l:asympt-1} (i), we deduce that
$(\Delta^{m-1} u_t)'(r)>0$ and $(\Delta^{m-1} u_s)'(r)>0$ for any
$r>0$ and their antiderivatives admit a finite limit as $r\to
+\infty$. This yields $(\Delta^{m-1} u_t)',(\Delta^{m-1} u_s)'\in
L^1(0,+\infty)$. Moreover by \eqref{eq:inequalities} we obtain
$$
\int_0^{\infty} (\Delta^{m-1} u_t)'(\sigma)\, d\sigma<
\int_0^{\infty} (\Delta^{m-1} u_s)'(\sigma)\, d\sigma
$$
and hence by Lemma \ref{l:asympt-1} (i)
\begin{align*}
& 0=\lim_{r\to +\infty} \Delta^{m-1}
u_t(r)=\gamma_t+\int_0^{\infty} (\Delta^{m-1} u_t)'(\sigma)\,
d\sigma\\
& <\gamma_s+\int_0^{\infty} (\Delta^{m-1} u_s)'(\sigma)\,
d\sigma=\lim_{r\to +\infty} \Delta^{m-1} u_s(r)=0 \, .
\end{align*}
We reached a contradiction.
\end{proof}

\section{Proof of Theorems
\ref{t:existence-odd}-\ref{t:estimates}}

{\it Proof of Theorem \ref{t:existence-odd}.} The proof of Theorem
\ref{t:existence-odd} is an immediate consequence of Lemma
\ref{odd}.

\bigskip

{\it Proof of Theorem \ref{t:existence-even}.} The proof of
Theorem \ref{t:existence-even} (i) is contained in Lemma
\ref{l:empty}. The proof of Theorem \ref{t:existence-even}
(ii)-(iii) is contained in Lemma \ref{l:non-empty}. The proof of
Theorem \ref{t:existence-even} (iv) is contained in Lemma
\ref{l:non-constant}.

\bigskip
{\it Proof of Theorem \ref{t:non-ex-even}.} The
proof follows closely the argument performed in the proof of
Theorem 1 in \cite{bffg}. Suppose by contradiction that
\eqref{eq_1} admits an entire solution $u$. From \eqref{eq_1} we
have that $u^{(2m-2)}$ is strictly convex and hence at least one
of the two limits ${\ds\lim_{x\to +\infty} u^{(2m-2)}(x)}$ or
${\ds\lim_{x\to -\infty} u^{(2m-2)}(x)}$ is equal to $+\infty$ and
up to replace $u$ with the $u(-x)$ we may assume the first one is
$+\infty$. After a finite number of iterations we deduce that
${\ds\lim_{x\to+ \infty}u(x)=+\infty}$ and in particular by
\eqref{eq_1} we also have that $u^{(2m)}$ and, in turn, also
$u^{(2m-1)}$ diverge to $+\infty$ as $x\to +\infty$. Hence there
exists $M>0$ such that
\begin{equation} \label{eq:x>M}
u^{(2m)}(x)=e^{u(x)}\ge (u(x))^2 \quad \text{and} \quad
u^{(2m-1)}(x)\ge 0 \qquad \text{for any } x>M \, .
\end{equation}
Since \eqref{eq_1} is an autonomous equation we may assume that
$M=0$.

As in \cite{bffg} we apply the test function
method developed in \cite{mp}. More precisely, fix $\rho>0$ and a
nonnegative function $\phi\in C^{2m}_c([0,\infty))$ such that
$$
\phi(r)=
\begin{cases}
1\quad & \mbox{for }x\in[0,\rho]\\
0\quad & \mbox{for }x\ge 2\rho\ .
\end{cases}
$$

In particular we have
\begin{align*}
& \phi(0)=1,\qquad \phi^{(k)}(0)=0, \ \ \text{for any } k\in
\{1,\dots,2m-1\}, \\
& \phi(2\rho)=0, \qquad \phi^{(k)}(2\rho)=0 \ \ \text{for any }
k\in \{1,\dots,2m-1\} \, .
\end{align*}
By \eqref{eq_1}, \eqref{eq:x>M} and integration by parts we obtain
\begin{equation}\label{eq5}
\int_\rho^{2\rho}\phi^{(2m)}(x)u(x)dx=\int_0^{2\rho}\phi^{(2m)}(x)u(x)dx\ge\int_0^{2\rho}(u(x))^2\phi(x)dx+u^{(2m-1)}(0)
\ge\int_0^{2\rho}(u(x))^2\phi(x)dx.
\end{equation}
Exploiting the Young inequality
$u\phi^{(2m)}=u\phi^{1/2}\frac{\phi^{(2m)}}{\phi^{1/2}}\le
\frac{1}{2} \left(
 u^2\phi+\frac{|\phi^{(2m)}|^2}{\phi}\right)$ by \eq{eq5} we infer
\begin{equation}\label{eq6}
\int_\rho^{2\rho}\frac{(\phi^{(2m)}(x))^2}{\phi(x)}dx\ge\int_0^\rho
(u(x))^2dx.
\end{equation}
We now choose $\phi(x)=\phi_\rho(x)=\phi_0(\frac{x}{\rho})$, where
$\phi_0\in C^{(2m)}_c([0,\infty))$, $\phi_0\ge0$ and
$$
\phi_0(\tau)=
\begin{cases}
1\quad & \mbox{for }\tau\in[0,1]\\
0\quad & \mbox{for }\tau\ge2\ .
\end{cases}
$$
As noticed in \cite{mp}, there exists a function $\phi_0$ in such
class satisfying moreover
$$\int_1^2\frac{(\phi_0^{(2m)}(\tau))^2}{\phi_0(\tau)}d\tau=:A<\infty.$$
Then, thanks to a change of variables in the integrals, \eq{eq6}
yields
$$
A\rho^{-4m+1}\!=\!\rho^{-4m+1}\!\!\int_1^2\frac{(\phi_0^{(2m)}(\tau))^2}{\phi_0(\tau)}d\tau\!=\!
\rho^{-4m}\!\!\int_\rho^{2\rho}\frac{(\phi_0^{(2m)}(\frac{x}{\rho}))^2}{\phi_0(\frac{x}{\rho})}dx
\!=\!\int_\rho^{2\rho}\frac{(\phi^{(2m)}(x))^2}{\phi(x)}dx
\ge\!\!\! \int_0^\rho (u(x))^2dx
$$
for any $\rho>0$. Letting $\rho\to\infty$, the previous inequality
contradicts the fact that $u$ diverges to $+\infty$ as $r\to
+\infty$.

\bigskip

{\it Proof of Theorem \ref{t:ex-odd}.} We follow
the idea performed in the proof of Theorem \ref{t:existence-odd}
for symmetric solutions. Since \eqref{eq_1} is an autonomous
equation we may assume that $u$ is solution of \eqref{eq_1}
defined in a neighborhood $I$ of $x=0$; we may assume that $I$ is
the maximal interval of continuation. We put $a_0:=u(0)$ and
$a_k:=u^{(k)}(0)$ for any $k\in \{1,\dots,2m-1\}$. Since
$u^{(2m)}=-e^u$ then $u^{(2m-1)}$ is decreasing and hence
$u^{(2m-1)}(x)\le a_{2m-1}$ for any $x\in I$, $x>0$. We then
define the unique solution of the Cauchy problem
\begin{equation} \label{eq:ww}
\begin{cases}
w^{(2m-1)}=a_{2m-1} \\
w^{(k)}(0)=a_k \qquad \text{for any } k\in \{0,\dots,2m-2\} \, .
\end{cases}
\end{equation}
We observe that $w$ is a polynomial and it is a global solution of \eqref{eq:ww}.
Then $u(x)\le w(x)$ for any $x\in I$, $x>0$ and if we assume by
contradiction that $I$ is bounded from above then $u$ would be
bounded from above and $e^u$ bounded in $I\cap \{x\in \R:x>0\}$.
In a standard way this brings to a contradiction with the
maximality of $I$. In a similar way one may prove left
continuation. This completes the proof of the first part.

Let $m=1$ so that \eqref{eq_1} becomes $-u''=e^u$.
Clearly this equation can be solved explicitly but here we want
only to show symmetry. From the first part of the proof of Lemma \ref{l:1-d} we know that
there exists $x_0\in \R$ such that $u'(x_0)=0$.
The proof of the symmetry now follows immediately since the function
$v(x)=u(2x_0-x)$ satisfies $-v''=e^v$ and $v'(x_0)=0$ and hence it
coincides with $u$ by uniqueness of the solution of a Cauchy
problem.

Finally we show that for $m\ge 3$ odd, equation
\eqref{eq_1} admits a nonsymmetric solution. It is enough to
consider the solution of the following Cauchy problem
\begin{equation} \label{eq:1-d-p}
\begin{cases}
-u^{(2m)}=e^u \\
u(0)=0\, , \qquad u'(0)=1\\
u^{(k)}(0)=0 \qquad \text{for any } k\in \{2,\dots,2m-1\} \, .
\end{cases}
\end{equation}
We recall that $u$ is a global solution of \eqref{eq:1-d-p} from
what we showed above. Suppose by contradiction that $u$ is
symmetric with respect to some $x_0\in \R$. Then $u^{(k)}(x_0)=0$
for any $k\in \{1,\dots,2m-1\}$ odd. But $u^{(2m-1)}$ is
decreasing and it equals to zero at $x=0$ so that $x=0$ is the
unique point where it vanishes. This implies $x_0=0$ and hence
$u'(0)=0$, a contradiction.

\bigskip

{\it Proof of Theorem \ref{t:estimates}.} The
proof of Theorem \ref{t:estimates} (i) is contained in Lemma
\ref{l:lim-0-cons}. The proof of Theorem \ref{t:estimates} (ii) is
contained in Lemma \ref{l:55}. The proof of Theorem
\ref{t:estimates} (iii) is contained in Lemma \ref{l:1-d}. The
proofs of Theorem \ref{t:estimates} (iv)-(v) are contained
respectively in Lemma \ref{l:asympt-1} and in Lemma
\ref{l:slow-decay}.

\section{Proof of Theorem \ref{t:non-ex-stable}}

Let $u$ be a stable solution of \eqref{eq_1}. We start by
considering the case $n<2m$. In this situation, we proceed
similarly to the proof of Theorem 6 in \cite{w}. We consider a
function $\eta\in C^\infty(\R^n)$ such that
\begin{equation} \label{eq:eta}
\eta=1 \quad \text{in } B_1\, , \qquad \eta=0 \quad \text{in }
\R^n\setminus B_2 \qquad \text{and } \qquad \|\eta\|_{L^\infty}\le
1
\end{equation}
and for any $R>0$ we define $\eta_R(x):=\eta(x/R)$. Then we have
\begin{equation} \label{eq:cutoff-even}
\int_{\R^n} |\Delta^{m/2} \eta_R|^2 dx=R^{n-2m} \int_{\R^n}
|\Delta^{m/2} \eta|^2 dx\to 0 \qquad \text{as } R\to +\infty
\end{equation}
if $m$ is even and
\begin{equation} \label{eq:cutoff-odd}
\int_{\R^n} |\nabla(\Delta^{\frac{m-1}2} \eta_R)|^2 dx=R^{n-2m}
\int_{\R^n} |\nabla(\Delta^{\frac{m-1}2} \eta)|^2 dx\to 0 \qquad
\text{as } R\to +\infty
\end{equation}
if $m$ is odd. Using $\eta_R$ as a test function in
\eqref{eq:stability} and exploiting
\eqref{eq:cutoff-even}-\eqref{eq:cutoff-odd} respectively in the
cases $m$ even and $m$ odd we infer
\begin{equation*}
\lim_{R\to +\infty} \int_{\R^n} e^u \eta_R^2 \, dx \le \lim_{R\to
+\infty} \int_{\R^n} |\Delta^{m/2} \eta_R|^2 dx=0
\end{equation*}
if $m$ is even and
\begin{equation*}
\lim_{R\to +\infty} \int_{\R^n} e^u \eta_R^2 \, dx \le \lim_{R\to
+\infty} \int_{\R^n} |\nabla(\Delta^{\frac{m-1}2} \eta_R)|^2 dx=0
\end{equation*}
if $m$ is odd. Therefore by Fatou Lemma , the fact that $\eta_R\to
1$ pointwise as $R\to +\infty$ and the stability of $u$, we obtain
$$
\int_{\R^n} e^u dx\le \lim_{R\to +\infty} \int_{\R^n} e^u \eta_R^2
\, dx=0
$$
for any $m\ge 1$ and this is absurd.

It remains to consider the case $n=2m$. Let $\eta$ be as in
\eqref{eq:eta}. We define the sequence of functions $\{\eta_k\}$
by putting
$$
\eta_k(x):=\frac 1k \sum_{j=k}^{2k-1} \eta\left(\frac
x{2^j}\right) \qquad \text{for any } k\ge 1 \, .
$$
Clearly $\eta_k\in C^\infty_c(\R^n)$ and hence it is an admissible
test function for \eqref{eq:stability}. We observe that if $m$ is
even, the functions $\Delta^{m/2} \eta(2^{-j}x)$ have supports
with zero measure intersections, i.e.
\begin{equation} \label{eq:measure-1}
\Big|{\rm supp} \, \Big(\Delta^{m/2} \eta(2^{-i}x)\Big)\cap {\rm
supp} \, \Big( \Delta^{m/2} \eta(2^{-j}x)\Big)\Big|=0 \qquad
\text{if } i\neq j \, .
\end{equation}
Similarly if $m$ is odd we have
\begin{equation} \label{eq:measure-2}
\Big|{\rm supp} \, \Big(|\nabla(\Delta^{\frac{m-1}2}
\eta(2^{-i}x))|\Big) \cap {\rm supp} \,
\Big(|\nabla(\Delta^{\frac{m-1}2} \eta(2^{-j}x))|\Big)\Big|=0
\qquad \text{if } i\neq j \, .
\end{equation}
By \eqref{eq:measure-1}-\eqref{eq:measure-2} and the fact that
$n=2m$ we have
\begin{align} \label{eq:cutoff-even-bis}
& \int_{\R^n} |\Delta^{m/2} \eta_k|^2 dx=\frac{1}{k^2}\int_{\R^n}
\left[\sum_{j=k}^{2k-1} 2^{-jm} \Delta^{m/2}
\eta(2^{-j}x)\right]^2dx=\frac{1}{k^2}\sum_{j=k}^{2k-1}
\int_{\R^n} 2^{-2jm} |\Delta^{m/2} \eta(2^{-j}x)|^2dx
\\
\notag & =\frac{1}{k^2} \sum_{j=k}^{2k-1} \int_{\R^n}
|\Delta^{m/2} \eta|^2dx=\int_{\R^n} |\Delta^{m/2} \eta|^2dx\cdot
\frac 1k \to 0 \qquad \text{as } k\to +\infty
\end{align}
if $m$ is even and
\begin{align} \label{eq:cutoff-odd-bis}
& \int_{\R^n} |\nabla(\Delta^{\frac{m-1}2} \eta_k)|^2
dx=\frac{1}{k^2}\int_{\R^n} \left|\sum_{j=k}^{2k-1} 2^{-jm}
\nabla(\Delta^{\frac{m-1}2} \eta(2^{-j}x))\right|^2dx
\\
\notag & =\frac{1}{k^2} \sum_{j=k}^{2k-1}
 \int_{\R^n} 2^{-2jm}|\nabla(\Delta^{\frac{m-1}2} \eta(2^{-j}x))|^2dx
 =\int_{\R^n} |\nabla(\Delta^{\frac{m-1}2} \eta)|^2dx\cdot \frac 1k \to 0
\qquad \text{as } k\to +\infty
\end{align}
if $m$ is odd. Moreover $\eta_k\to 1$ pointwise as $k\to +\infty$.
Therefore by \eqref{eq:cutoff-even-bis}, \eqref{eq:cutoff-odd-bis}
respectively in the cases $m$ even and $m$ odd, Fatou Lemma and
the stability of $u$, we obtain
\begin{align*}
\int_{\R^n} e^u dx\le\lim_{k\to +\infty} \int_{\R^n} e^u \eta_k^2
\, dx=0
\end{align*}
and this is absurd.

\section{Proof of Theorems
\ref{t:preliminar-1}-\ref{t:Hardy-Rellich-2} and Proposition \ref{p:HR-1d}}
Let
\begin{equation*}
\Phi:\R^n\setminus\{0\}\to \mathcal C:=\R\times \SN\subset
\R^{n+1}
\end{equation*}
be the diffeomorphism defined by
\begin{equation*}
\Phi(x):=\left(-\log|x|,\frac{x}{|x|}\right) \qquad \text{for any
} x\in \R^n\setminus\{0\}
\end{equation*}
and let $\mathcal C_\Omega:=\Phi(\Omega)\subseteq \mathcal C$ for
any open set $\Omega\subseteq \R^n \setminus\{0\}$. For any
$\alpha\in \R$ let us introduce the linear operator
\begin{equation*}
T_\alpha:C^\infty_c (\Omega)\to C^\infty_c(\mathcal C_\Omega)
\end{equation*}
by
\begin{equation} \label{eq:Tu}
T_\alpha
\varphi(t,\theta):=e^{\frac{4-n-\alpha}2t}\varphi(e^{-t}\theta)
\qquad \text{for any } (t,\theta)\in \mathcal C_\Omega \quad
\text{and } \varphi\in C^\infty_c (\Omega) \, .
\end{equation}
Clearly $T_\alpha$ is an isomorphism between vector spaces.  Let
us denote by $\mu$ the volume measure on $\mathcal C$.

\begin{lemma} \label{l:preliminar-0}
Let $n\ge 2$. For any $R>1$ put $\Omega_R:=\R^n\setminus \overline
B_R$. Let $\varphi\in C^\infty_c(\Omega_R)$, $\alpha\in\R$ and
$\beta\ge 0$. Then
\begin{align} \label{eq:stimona}
& \int_{\Omega_R} \frac{|x|^{\alpha} |\Delta
\varphi|^2}{(\log|x|)^\beta} \, dx=\int_{\mathcal C_{\Omega_R}}
|t|^{-\beta} |Lw(t,\theta)|^2 \,  d\mu+ \int_{\mathcal
C_{\Omega_R}} |t|^{-\beta} \left[(\partial^2_t
w(t,\theta))^2+2\bar
\gamma_{n,\alpha} (\partial_t w(t,\theta))^2\right]d\mu\\
\notag & +\int_{\mathcal C_{\Omega_R}} |t|^{-\beta}\left[
2|\nabla_{\SN}(\partial_t
w(t,\theta))|^2-\beta(\beta+1) |t|^{-2}|\nabla_{\SN} w(t,\theta)|^2\right] d\mu\\
\notag & +\int_{\mathcal C_{\Omega_R}} |t|^{-\beta} \left[
(\alpha-2)\beta|t|^{-1} |\nabla_{\SN}
w(t,\theta)|^2-\beta(\beta+1)\gamma_{n,\alpha}|t|^{-2} w^2(t,\theta) \right] d\mu\\
\notag & +\int_{\mathcal C_{\Omega_R}} |t|^{-\beta}\left[
(\alpha-2)\beta \gamma_{n,\alpha} |t|^{-1}
w^2(t,\theta)-(\alpha-2)\beta|t|^{-1} (\partial_t w(t,\theta))^2
\right] d\mu
\end{align}
where $w=T_\alpha \varphi\in C^\infty_c(\mathcal C_{\Omega_R})$,
$L=-\Delta_{\SN}+\gamma_{n,\alpha}$, $\gamma_{n,\alpha}$ is as in
Proposition \ref{p:Hardy-Rellich} and $\bar
\gamma_{n,\alpha}=\left(\frac{n-2}2\right)^2+\left(\frac{\alpha-2}2\right)^2$.
\end{lemma}

\begin{proof} Proceeding as in the proof of Lemma 2.4 in
\cite{calmus} we obtain
$$
\Delta
\varphi(x)=|x|^{-\frac{n+\alpha}2}[-Lw(-\log|x|,x/|x|)+\partial^2_t
w(-\log|x|,x/|x|)+(\alpha-2)\partial_t w(-\log|x|,x/|x|)]
$$
and hence
\begin{align*}
& \int_{\Omega_R} \frac{|x|^{\alpha} |\Delta
\varphi|^2}{(\log|x|)^\beta} \, dx=\int_{\mathcal C_{\Omega_R}}
|t|^{-\beta}[-Lw+\partial^2_t w+(\alpha-2)\partial_t w]^2
d\mu \\
& =\int_{\mathcal C_{\Omega_R}} |t|^{-\beta} \left[\,
|Lw|^2+(\partial^2_t w)^2+(\alpha-2)^2(\partial_t
w)^2+2(\partial^2_t w)(\Delta_{\SN} w)\right]d\mu
\\ & \quad +\int_{\mathcal C_{\Omega_R}} |t|^{-\beta} \left[2(\alpha-2)(\partial_t w)(\Delta_{\SN} w)-2\gamma_{n,\alpha}
w\, \partial^2_t w-2\gamma_{n,\alpha}(\alpha-2) w\, \partial_t
w+2(\alpha-2) \partial_t w \, \partial^2_t w \right]d\mu \, .
\end{align*}
The conclusion of the lemma then follows from the following
identities obtained after some integrations by parts
\begin{multline*}
2 \int_{\mathcal C_{\Omega_R}} |t|^{-\beta} (\partial^2_t
w)(\Delta_{\SN} w) \, d\mu=2\int_{\mathcal C_{\Omega_R}}
|t|^{-\beta} |\nabla_{\SN}(\partial_t w)|^2
d\mu-\beta(\beta+1)\int_{\mathcal C_{\Omega_R}}|t|^{-\beta-2}
|\nabla_{\SN} w|^2 d\mu \, , \\
2(\alpha-2)\int_{\mathcal C_{\Omega_R}} |t|^{-\beta} (\partial_t
w)(\Delta_{\SN} w) \, d\mu=(\alpha-2)\beta \int_{\mathcal
C_{\Omega_R}} |t|^{-\beta-1} |\nabla_{\SN} w|^2 d\mu \, ,\\
-2\gamma_{n,\alpha} \int_{\mathcal C_{\Omega_R}} |t|^{-\beta} w\,
\partial^2_t w \, d\mu=2\gamma_{n,\alpha} \int_{\mathcal C_{\Omega_R}} |t|^{-\beta} (\partial_t
w)^2 d\mu-\beta(\beta+1)\gamma_{n,\alpha} \int_{\mathcal
C_{\Omega_R}} |t|^{-\beta-2} w^2 d\mu \, ,
\\ -2\gamma_{n,\alpha}(\alpha-2) \int_{\mathcal C_{\Omega_R}}|t|^{-\beta} w\, \partial_t
w \, d\mu=(\alpha-2)\beta \gamma_{n,\alpha} \int_{\mathcal
C_{\Omega_R}} |t|^{-\beta-1} w^2 d\mu \, ,
\\
2(\alpha-2) \int_{\mathcal C_{\Omega_R}}  |t|^{-\beta} \partial_t
w \,
\partial^2_t w \, d\mu
=-(\alpha-2)\beta \int_{\mathcal C_{\Omega_R}} |t|^{-\beta-1}
(\partial_t w)^2 d\mu \, . \qquad \qquad \qquad
\end{multline*}
\end{proof}

The next three lemmas are devoted to suitable integral
inequalities involving functions in $H^2(\SN)$.

We start with the following inequality obtained with an
integration by parts:
\begin{equation} \label{eq:Poincare-2nd-ord}
\int_{\SN} |\nabla_{\SN} \psi|^2 dS\le \frac 12
\left(\|\Delta_{\SN}
\psi\|_{L^2(\SN)}^2+\|\psi\|_{L^2(\SN)}^2\right) \qquad \text{for
any } \psi\in H^2(\SN) \, .
\end{equation}

We recall from \cite[Proposition 1.1]{calmus} the following
estimate

\begin{lemma} \label{l:sfera-1}
\  \cite{calmus} Let $n\ge 2$. Let $\alpha\in \R$ and let
$\gamma_{n,\alpha}$ and $L$ be as in Lemma \ref{l:preliminar-0}.
Then
\begin{equation} \label{eq:Poincare-L}
\int_{\SN} |L\psi|^2 dS\ge \mu_{n,\alpha} \int_{\SN} |\psi|^2 dS
\qquad \text{for any } \psi\in H^2(\SN) \, ,
\end{equation}
where $\mu_{n,\alpha}$ is defined by \eqref{mu-n-alpha}.
\end{lemma}

When $\mu_{n,\alpha}=0$ estimate \eqref{eq:Poincare-L} becomes
trivial but using the argument performed in \cite[Proposition
1.1]{calmus} one deduces the following estimate

\begin{lemma} \label{l:sfera-1-bis}
Let $n\ge 2$ and $\alpha\in\R$ be such that $\mu_{n,\alpha}=0$
with $\mu_{n,\alpha}$ and $\gamma_{n,\alpha}$ as in Proposition
\ref{p:Hardy-Rellich}. Let $L$ be as in Lemma \ref{l:sfera-1}. Let
$\bar j\in \N\cup\{0\}$ be such that
$0=\mu_{n,\alpha}=|\gamma_{n,\alpha}+\bar j(n-2+\bar j)|^2$ and
define $\bar\mu_{n,\alpha}:=\ds{\min_{j\in \N\cup\{0\}, j\neq \bar
j} |\gamma_{n,\alpha}+j(n-2+j)|^2}>0$. Finally let $V$ be the
eigenspace of $-\Delta_{\SN}$ corresponding to the eigenvalue
$\bar j(n-2+\bar j)$. Then
\begin{equation} \label{eq:Poincare-L-barrata}
\int_{\SN} |L\psi|^2 dS\ge \bar\mu_{n,\alpha} \int_{\SN} |\psi|^2
dS \qquad \text{for any } \psi\in V^{\perp}  \, .
\end{equation}
\end{lemma}

The next lemma is devoted to an estimate for the $L^2(\SN)$-norm
of the gradient.

\begin{lemma} \label{l:sfera-2} Let $n\ge 2$, $\alpha\in\R$ and let $L$,
$\gamma_{n,\alpha}$ and $\mu_{n,\alpha}$ be as in Lemma
\ref{l:sfera-1}.
\begin{itemize}
\item[(i)] If $\mu_{n,\alpha}>0$ then
\begin{equation*}
\int_{\SN} |\nabla_{\SN} \psi|^2 dS\le \left[1+\mu_{n,\alpha}^{-1}
(\gamma_{n,\alpha}^2+1/2)\right] \int_{\SN} |L\psi|^2 dS \qquad
\text{for any } \psi\in H^2(\SN) \, .
\end{equation*}

\item[(ii)] If $\mu_{n,\alpha}=0$ let $\bar j$ be the unique value
of $j\in \N\cup \{0\}$ for which the minimum in \eqref{mu-n-alpha}
is achieved and put $\bar\mu_{n,\alpha}:=\ds{\min_{j\in
\N\cup\{0\}, j\neq \bar j} |\gamma_{n,\alpha}+j(n-2+j)|^2}>0$.
Then
\begin{equation*}
\int_{\SN} |\nabla_{\SN} \psi|^2 dS\le
\left[1+\bar\mu_{n,\alpha}^{-1} (\gamma_{n,\alpha}^2+1/2)\right]
\int_{\SN} |L\psi|^2 dS+|\gamma_{n,\alpha}| \int_{\SN} \psi^2 dS
\end{equation*}
for any $\psi\in H^2(\SN)$.
\end{itemize}

\end{lemma}

\begin{proof} Let start with the proof of (i). By
\eqref{eq:Poincare-2nd-ord}, \eqref{eq:Poincare-L} we have
\begin{align*}
& \int_{\SN} |\nabla_{\SN} \psi|^2 dS\le \frac 12 \left(\int_{\SN}
|L\psi-\gamma_{n,\alpha}\psi|^2 dS+\int_{\SN} \psi^2
dS\right)\\
& \le \frac 12 \left(2\int_{\SN} |L\psi|^2
dS+(2\gamma_{n,\alpha}^2+1)\int_{\SN} \psi^2 dS\right)\le \frac 12
\left[2+\mu_{n,\alpha}^{-1} (2\gamma_{n,\alpha}^2+1)\right]
\int_{\SN} |L\psi|^2 dS
\end{align*}
for any $\psi\in H^2(\SN)$ thus completing the proof of (i).

Let us proceed with the proof of (ii). Let $V$ be as in the
statement of Lemma \ref{l:sfera-1-bis} and for any $\psi\in
H^2(\SN)$ let $\psi_1\in V$ and $\psi_2\in V^{\perp}$ be such that
$\psi=\psi_1+\psi_2$. Finally put $\lambda_{\bar j}=\bar
j(n-2+\bar j)=-\gamma_{n,\alpha}$.

Then by \eqref{eq:Poincare-2nd-ord} and
\eqref{eq:Poincare-L-barrata} we have
\begin{align*}
& \int_{\SN} |\nabla_{\SN} \psi|^2 dS=\int_{\SN} |\nabla_{\SN}
\psi_1|^2 dS+\int_{\SN} |\nabla_{\SN} \psi_2|^2 dS\\
& \quad \le \lambda_{\bar j} \int_{\SN} \psi_1^2 dS+\frac
12\left(\int_{\SN} |L\psi_2-\gamma_{n,\alpha}\psi_2|^2
dS+\int_{\SN} \psi_2^2 dS \right)\\
& \quad \le \lambda_{\bar j} \int_{\SN} \psi^2 dS+\frac 12\left(
2\int_{\SN} |L\psi_2|^2 dS+(2\gamma_{n,\alpha}^2+1)\int_{\SN}
\psi_2^2 dS \right) \\
& \quad \le -\gamma_{n,\alpha} \int_{\SN} \psi^2 dS+[1+\bar
\mu_{n,\alpha}^{-1}(\gamma_{n,\alpha}^2+1/2)] \int_{\SN}
|L\psi_2|^2 dS
\end{align*}
and the conclusion follows since $\int_{\SN} |L\psi|^2
dS=\int_{\SN} |L\psi_1|^2 dS+\int_{\SN} |L\psi_2|^2 dS$.
\end{proof}

{\it End of the proof of Theorem \ref{t:preliminar-1}.} Let
$w:=T_\alpha \varphi$. By \eqref{eq:stimona}, Lemmas
\ref{l:sfera-1}-\ref{l:sfera-2} and the fact that $\alpha\le 0$,
$\gamma_{n,\alpha}<0$ being $\mu_{n,\alpha}=0$, we obtain
\begin{align*}
& \int_{\R^n\setminus \overline B_R} \frac{|x|^{\alpha} |\Delta
\varphi|^2}{(\log|x|)^\beta} \, dx \ge \int_{\mathcal
C_{\Omega_R}} |t|^{-\beta} |Lw(t,\theta)|^2 \,  d\mu\\
&  + \int_{\mathcal C_{\Omega_R}} |t|^{-\beta} \left[(\partial^2_t
w(t,\theta))^2+2\bar \gamma_{n,\alpha} (\partial_t
w(t,\theta))^2+2|\nabla_{\SN}(\partial_t
w(t,\theta))|^2\right]d\mu\\
& -C(n,\alpha,\beta)\left[\int_{\mathcal
C_{\Omega_R}}|t|^{-2-\beta} |Lw(t,\theta)|^2d\mu+\int_{\mathcal
C_{\Omega_R}}|t|^{-1-\beta} |Lw(t,\theta)|^2d\mu \right] \\
& +(2-\alpha)\beta \int_{\mathcal C_{\Omega_R}}
|t|^{-1-\beta} (\partial_t w(t,\theta))^2 d\mu \\
&  \ge [1-C(n,\alpha,\beta)((\log R)^{-2}+(\log R)^{-1})]
\int_{\mathcal C_{\Omega_R}} |t|^{-\beta} |Lw(t,\theta)|^2 \, d\mu
+2\bar \gamma_{n,\alpha} \int_{\mathcal C_{\Omega_R}}
|t|^{-\beta}(\partial_t w(t,\theta))^2 d\mu \, ,
\end{align*}
where $C(n,\alpha,\beta)$ is a positive constant depending only on
$n$, $\alpha$ and $\beta$. If choose $R$ sufficiently large the
constant $[1-C(n,\alpha,\beta)((\log R)^{-2}+(\log R)^{-1})]$
becomes positive so that using the one dimensional weighted Hardy
inequality
$$
\left(\frac{\beta+1}2\right)^2 \int_0^\infty t^{-\beta-2}
(\eta(t))^2 dt\le \int_0^\infty t^{-\beta} (\eta'(t))^2 dt \qquad
\text{for any } \eta\in C^\infty_c(0,+\infty)
$$
we obtain
\begin{align*}
& \int_{\R^n\setminus \overline B_R} \frac{|x|^{\alpha} |\Delta
\varphi|^2}{(\log|x|)^\beta} \, dx\ge 2\bar\gamma_{n,\alpha}
\left(\frac{\beta+1}2\right)^2 \int_{\mathcal C_{\Omega_R}}
|t|^{-\beta-2} (w(t,\theta))^2 d\mu\\
& \quad =2\bar\gamma_{n,\alpha} \left(\frac{\beta+1}2\right)^2
\int_{\R^n\setminus \overline B_R} \frac{|x|^{\alpha-4}
\varphi^2}{(\log|x|)^{\beta+2}} \, dx \, .
\end{align*}
This completes the proof of the theorem.

{\it Proof of Proposition \ref{p:HR-1d}.} It is enough to prove \eqref{eq:H-1d-0}. Let $\varphi\in C^\infty_c(\R\setminus \{0\})$ and $\alpha\in\R$. By integration by parts we have that
\begin{equation*}
\int_\R |x|^{\alpha-2} x \varphi(x)\varphi'(x)\, dx=-\frac{\alpha-1}2 \int_\R |x|^{\alpha-2} (\varphi(x))^2 dx
\end{equation*}
and hence by H\"older inequality it follows
\begin{equation*}
\frac{|\alpha-1|}2 \int_\R |x|^{\alpha-2} (\varphi(x))^2 dx\le \left(\int_\R |x|^{\alpha-2} (\varphi(x))^2 dx\right)^{1/2} \left(\int_\R |x|^{\alpha} (\varphi'(x))^2 dx\right)^{1/2} \, .
\end{equation*}
This completes the proof of \eqref{eq:H-1d-0}.

\bigskip

{\it End of the proof of Theorem \ref{t:Hardy-Rellich-2}.} The
proof of \eqref{eq:ineq-1} follows by using Proposition
\ref{p:Hardy-Rellich} and Theorem \ref{t:preliminar-1} and taking
$R>1$ large enough. The proof of \eqref{eq:ineq-2} follows by
combining Proposition \ref{p:Hardy-Rellich}, Theorem
\ref{t:preliminar-1} with the second order Hardy-type inequality
\begin{equation*}
\frac 14 \int_{\R^2\setminus B_R} \frac{\varphi^2}{|x|^2\log^2
|x|}\, dx\le \int_{\R^2\setminus B_R} |\nabla \varphi|^2 dx \qquad
\text{for any } \varphi\in C^\infty_c(\R^2\setminus \overline
B_R)\ R>1
\end{equation*}
(see \cite{adimurthi} and \cite[proof of Theorem 3]{farina}) and
the classical Hardy inequality in dimension $n\ge 3$ and taking
$R>1$ large enough.

\section{Proof of Theorems
\ref{t:n<=2m-1}-\ref{t:conformal-d}} \label{s:stability-proof}

Let $u$ be a solution of \eqref{eq_1} satisfying
the assumptions of Theorems \ref{t:n<=2m-1}-\ref{t:staocs-even}.
Then by Theorem \ref{t:estimates} (i)-(iv), there exist
$C,\overline r>0$ such that
\begin{equation*}
u(x)<-C |x| \qquad \text{for any } |x|>\overline r \, .
\end{equation*}
In particular we have that
\begin{equation*}
e^{u(x)}<e^{-C |x|} \qquad \text{for any } |x|>\overline r \, .
\end{equation*}
According with
\eqref{eq:Hardy-Rellich-1}-\eqref{eq:Hardy-Rellich-2} and
\eqref{eq:HR-1d-ter}-\eqref{eq:ineq-2}, we define the radial
function in the following different cases
\begin{align*}
& V(x):= \\
&
\begin{cases}
 \left(\prod_{i=0}^{m/2-1} \bar \gamma_{n,-4i}
\right)\!\!\cdot\!\! \left(\prod_{i=0}^{m/2-1}
\Big(\frac{2i+1}2\Big)^2 \right)
 \frac{2^{m/2}}{|x|^{2m}(\log|x|)^{m}}
 & m \ \text{even,} \ 3\le n \le 2m \ \text{even} \, , \\[10pt]
 \left(\prod_{i=0}^{\frac{m-1}2-1} \bar
\gamma_{n,-4i-2} \right)\!\!\cdot \!\!
\left(\prod_{i=0}^{\frac{m-1}2-1} \Big(\frac{2i+3}2\Big)^2
\right)\!\! \frac{2^{\frac{m-1}2-2}}{|x|^{2m}(\log|x|)^{m+1}}
 & m\ge 3 \ \text{odd,} \ 2\le n\le 2m \ \text{even} \, , \\[10pt]
 \left(\prod_{i=0}^{\frac{m-1}2-1} (4i-3)^2(4i-5)^2\right)
\frac{2^{-2m}}{|x|^{2m}}  &  m\ge 1 \ \text{odd,} \ n=1 \, , \\[10pt]
\left(\prod_{i=1}^{m/2} \mu_{n,\alpha_i}  \right)
\frac{1}{|x|^{2m}} & m \ \text{even,} \ 3\le n\le 2m \
\text{odd or} \ n>2m \, , \\[10pt]
\left(\frac{n-2}2\right)^2 \left(\prod_{i=1}^{\frac{m-1}2}
\mu_{n,\alpha_i} \right) \frac{1}{|x|^{2m}} & m \ \text{odd,} \
3\le n\le 2m \ \text{odd or} \ n>2m \, .
\end{cases}
\end{align*}
where $\bar \gamma_{n,-4i}$ and $\bar \gamma_{n,-4i-2}$ are
defined in Theorem \ref{t:preliminar-1}.

We observe that the function $x\mapsto
\frac{e^{-C |x|}}{V(x)}$ vanishes as $|x|\to +\infty$. Therefore
the exists $R>\overline r$ such that
\begin{equation*}
\frac{e^{-C |x|}}{V(x)}<1 \qquad \text{for any } |x|>R
\end{equation*}
and hence by \eqref{eq:Hardy-Rellich-1}-\eqref{eq:Hardy-Rellich-2}
and \eqref{eq:HR-1d-ter}-\eqref{eq:ineq-2}, up to enlarging $R$,
we obtain
\begin{align*}
& \int_{\R^n\setminus \overline B_R} |\Delta^{m/2} \varphi|^2
dx-\int_{\R^n\setminus \overline B_R} e^u \varphi^2 dx \ge
\int_{\R^n\setminus \overline B_R} |\Delta^{m/2} \varphi|^2
dx-\int_{\R^n\setminus \overline B_R} \frac{e^{-C |x|}}{V(x)} \,
V(x)\varphi^2(x) \, dx \\
& \qquad \ge \int_{\R^n\setminus \overline B_R} |\Delta^{m/2}
\varphi|^2 dx-\int_{\R^n\setminus \overline B_R} V(x)\varphi^2(x)
\, dx\ge 0 \qquad \text{for any } \varphi\in
C^\infty_c(\R^n\setminus \overline B_R)
\end{align*}
if $m$ is even and
\begin{align*}
& \int_{\R^n\setminus \overline B_R} |\nabla(\Delta^{\frac{m-1}2}
\varphi)|^2 dx-\!\! \int_{\R^n\setminus \overline B_R} e^u
\varphi^2 dx \ge \int_{\R^n\setminus \overline B_R}
|\nabla(\Delta^{\frac{m-1}2} \varphi)|^2 dx-\!\!
\int_{\R^n\setminus \overline B_R} \frac{e^{-C |x|}}{V(x)} \,
V(x)\varphi^2(x) \, dx \\
& \qquad \ge \int_{\R^n\setminus \overline B_R}
|\nabla(\Delta^{\frac{m-1}2} \varphi)|^2 dx-\int_{\R^n\setminus
\overline B_R} V(x)\varphi^2(x) \, dx\ge 0 \qquad \text{for any }
\varphi\in C^\infty_c(\R^n\setminus \overline B_R) \, .
\end{align*}
This completes the proof of the first four theorems.

The proof of Theorem \ref{t:conformal-d} follows in the same way
as above since the explicit solution $u$ defined in
\eqref{eq:sol-esplicita} satisfies $u(x)=-4m\log|x|+O(1)$ as
$|x|\to +\infty$.

\section{An autonomous equation associated with \eqref{Cauchy}}
\label{s:dynamical}

In order to provide detailed information on the asymptotic
behavior at infinity of radial solutions of polyharmonic equations
like \eqref{eq_1}, it can be useful to reduce the equation in
\eqref{Cauchy} to an autonomous equation by mean of a suitable
change of variable, see for example
\cite{agg,bffg,fergru,fergrukar,ferwar,gazgru} where biharmonic
equations with both power and exponential type nonlinearities are
studied.

Throughout this section we will assume that $n>2m$. Consider the
function $u_S(x)=-2m \log|x|$ for any $x\neq 0$. By direct
computation one sees that $u_S$ solves the equation
\begin{equation} \label{eq:u-S}
(-\Delta)^m u_S=\lambda_S e^{u_S} \qquad \text{in }
\R^n\setminus\{0\}
\end{equation}
where ${\ds\lambda_S=2^m m! \prod_{k=1}^{m} (n-2k)} >0$. In order
to find a solution of \eqref{eq_1} it is sufficient to define the
function $U_S(x)=u_S(x)+\log \lambda_S$ for any $x\neq 0$ which
clearly satisfies
\begin{equation} \label{eq:U}
(-\Delta)^m U_S=e^{U_S} \qquad \text{in } \R^n\setminus\{0\} \, .
\end{equation}
Then we put $s=\log r$ in such a way that if $u=u(r)$ is a radial
solution of \eqref{eq_1} then the function
\begin{equation} \label{eq:cambio}
w(s):=u(e^s)-U_S(e^s)=u(e^s)+2ms-\log \lambda_S
\end{equation}
solves the equation
\begin{equation} \label{eq:autonomous}
Q_m(\partial_s)w(s)=\lambda_S (e^{w(s)}-1) \, ,\qquad s\in \R
\end{equation}
where $U_S$ and $\lambda_S$ are as in \eqref{eq:u-S}-\eqref{eq:U},
$Q_m$ is the polynomial of degree $2m$ defined by
\begin{equation*}
Q_m(t):=(-1)^m \prod_{j=0}^{m-1} (t-2j)(t+n-2j-2)
\end{equation*}
and $Q_m(\partial_s)$ is the linear differential operator of order
$2m$ whose characteristic polynomial is given by $Q_m$.

We observe that equation \eqref{eq:autonomous} admits the trivial
solution $w\equiv 0$ and according to the change of variable
\eqref{eq:cambio}, $w$ corresponds to the function $u(r)=-2m\log
r+\log\lambda_S$. For this reason it may be interesting to study
the behavior of solutions of \eqref{eq:autonomous} approaching to
zero as $s\to +\infty$. To this purpose it may be useful to
consider the linearized equation at $w=0$ corresponding to
\eqref{eq:autonomous}:
\begin{equation*}
Q_m(\partial_s)w(s)=\lambda_S w(s) \, ,\qquad s\in \R \, .
\end{equation*}
The last equation may be rewritten as $P_m(\partial_s) w=0$ once
we define
\begin{equation*}
P_m(t):=Q_m(t)-\lambda_S
\end{equation*}
and we denote by $P_m(\partial_s)$ the linear differential
operator whose characteristic polynomial is given by $P_m$.

In order to have a clear picture on the behavior of solutions of
\eqref{eq:autonomous}, a fundamental aspect that has to be taken
in consideration, is the presence or not of non real roots of the
polynomial $P_m$.

We also recall that in the cases $m=1$ and $m=2$ the condition
which determines the presence or not of non real roots of $P_m$
determines also the existence and nonexistence of stable solutions
of \eqref{eq_1}, see for example \cite{BrezisVazquez,farina,jl}
for the case $m=1$ and \cite{bffg,dggw} for the case $m=2$.

By direct computation one may check that if $m=1$ and $n>2$ then
$P_m$ admits non real roots if and only if $n\le 9$ and from
\cite{farina,jl} we know that if $n\le 9$ then \eqref{eq_1} does
not admit any stable solution (also between non radial solutions)
while if $n\ge 10$ all radial entire solutions of \eqref{eq_1} are
stable. Similarly if $m=2$ and $n>4$ then $P_m$ admits non real
roots if and only if $n\le 12$ and from \cite{bffg,ddgm} we know
that if $n\le 12$ then \eqref{eq_1} admits radial entire solutions
which are unstable while if $n\ge 13$ all radial entire solutions
of \eqref{eq_1} are stable.

We also observe that in both cases $m=1$ and $m=2$ the existence
of non real roots is strictly related to the values taken by the
parameter $\lambda_S$ and the best constant for the corresponding
Hardy-Rellich inequality as the dimension $n$ varies. Indeed
stability of all radial entire solutions occurs if and only if
\begin{equation} \label{eq:m12}
\begin{array}{ll}
2(n-2)=\lambda_S\le \frac{(n-2)^2}4 & \qquad \text{if } m=1 \, , \\
8(n-2)(n-4)=\lambda_S\le \frac{n^2(n-4)^2}{16} & \qquad \text{if }
m=2 \, .
\end{array}
\end{equation}
See Proposition \ref{p:mitidieri} for the values of the optimal
constant in the Hardy-Rellich inequalities. The two inequalities
in \eqref{eq:m12} are equivalent respectively to $n\ge 10$ if
$m=1$ and $n\ge 13$ if $m=2$.

One may ask whether at least for $m\ge 4$ even and $n>2m$,
existence of radial unstable solutions of \eqref{eq_1} and/or
existence of non real roots of $P_m$, is again equivalent to the
validity of the inequality $\lambda_S>A_{n,m/2}$ with $A_{n,m/2}$
as in Proposition \ref{p:mitidieri}.

A question then arises: for any $m\ge 4$ even, does it exist a
critical dimension $n^*\in \N$ such that $\lambda_S>A_{n,m/2}$ for
any $2m<n\le n^*-1$ and $\lambda_S\le A_{n,m/2}$ for any $n\ge
n^*$? The next proposition answer positively to this question.

\begin{proposition} Let $m\ge 2$ even and let $\lambda_S$ and
$A_{n,m/2}$ respectively as in \eqref{eq:u-S} and in Proposition
\ref{p:mitidieri}. Then there exists $n^*\in \N$ such that
$\lambda_S>A_{n,m/2}$ for any $2m<n\le n^*-1$ and $\lambda_S\le
A_{n,m/2}$ for any $n\ge n^*$.
\end{proposition}

\begin{proof} First we observe that for any $n>2m$ we
may write
\begin{equation*}
\frac{A_{n,m/2}}{\lambda_S}=\frac{1}{8^m \cdot m!} \, \cdot
\prod_{i=0}^{m/2-1} (n+2m-4i-4)^2 \cdot  \prod_{i=1}^{m/2}
\frac{n-4i}{n-4i+2} \, .
\end{equation*}
Hence for any fixed $m$ the previous quotient is increasing with
respect to $n$ provided that $n>2m$.

For $n=2m+1$ the quotient becomes
\begin{align*}
\frac{A_{n,m/2}}{\lambda_S}&=\frac{1}{8^m \cdot m!} \, \cdot
\prod_{i=0}^{m/2-1} (4m-4i-3)^2 \cdot  \prod_{i=1}^{m/2}
\frac{2m-4i+1}{2m-4i+3}<\frac{1}{8^m \cdot m!} \, \cdot
\prod_{i=0}^{m/2-1} (4m-4i)^2\\
& < \frac{1}{8^m \cdot m!} \, 2^m \prod_{i=0}^{m-1} (2m+1-i)
=\frac{1}{8^m \cdot m!} \, 2^m \frac{2m+1}{m+1}\,
\frac{(2m)!}{m!}<2\cdot 4^{-m} \cdot \frac{(2m)!}{(m!)^2}=:a_m
\end{align*}
Since the sequence $\{a_m\}$ is decreasing then we obtain
$$
\frac{A_{n,m/2}}{\lambda_S}<a_2=\frac 34<1 \, .
$$
On the other hand it is clear that for any $m$ fixed we have that
$\lim_{n\to +\infty} \frac{A_{n,m/2}}{\lambda_S}=+\infty$.

After collecting all the above information the proof of the
proposition follows.
\end{proof}

In contrast with the case $m=2$, numerical evidence shows that for
$m\ge 4$ even, the condition $n\ge n^*$ is not sufficient to
guarantee that all the roots of $P_m$ are real, see the two pictures
below respectively in the cases $m=4$, $n=n^*-1=17$ and $m=4$, $n=n^*=18$.

\begin{figure}[th]
\begin{center}
{\includegraphics[height=60mm, width=120mm]{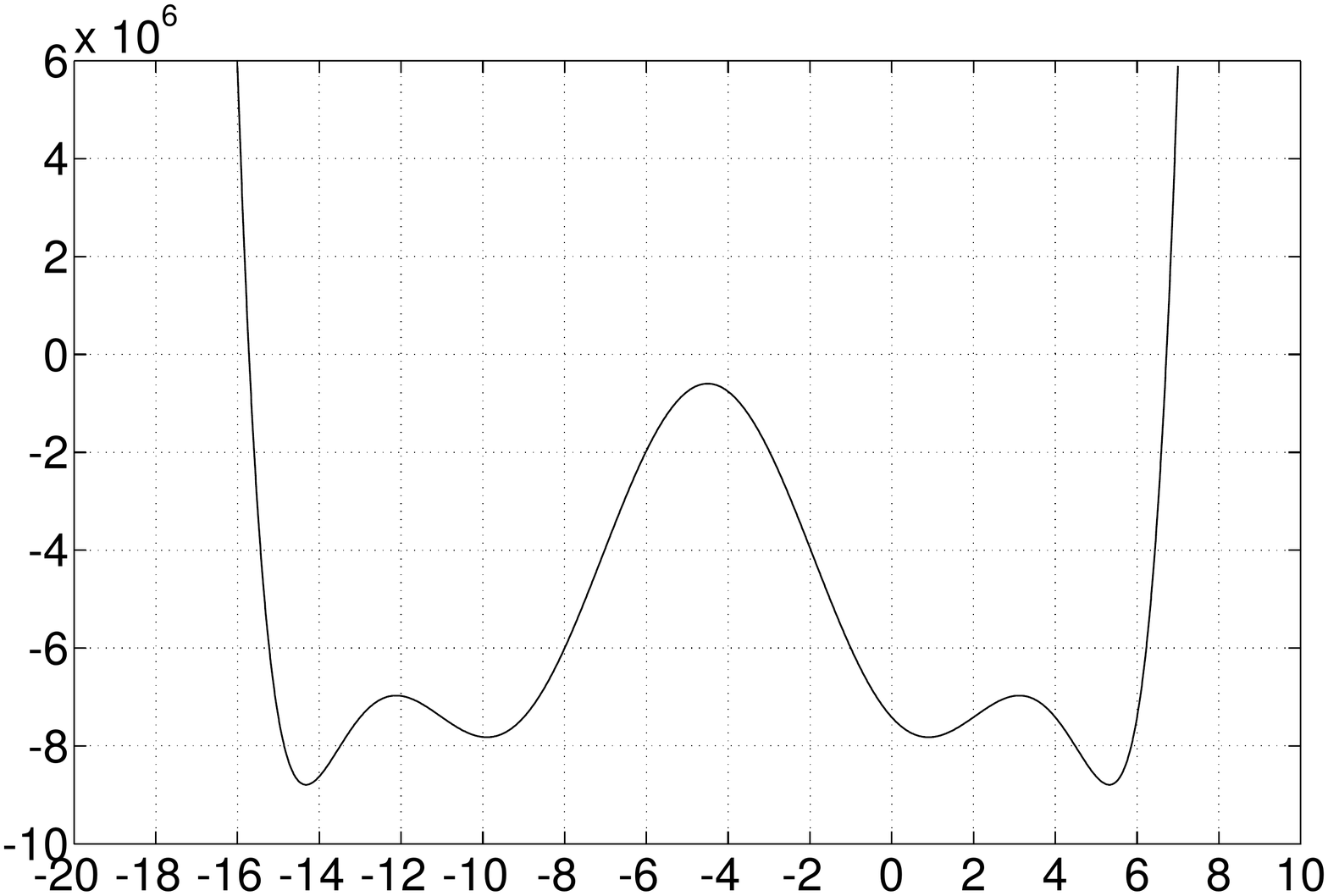}}
\caption{Graph of $P_m$ for $m=4$ and $n=n^*-1=17$ .}
\end{center}
\end{figure}

\begin{figure}[th]
\begin{center}
{\includegraphics[height=60mm, width=120mm]{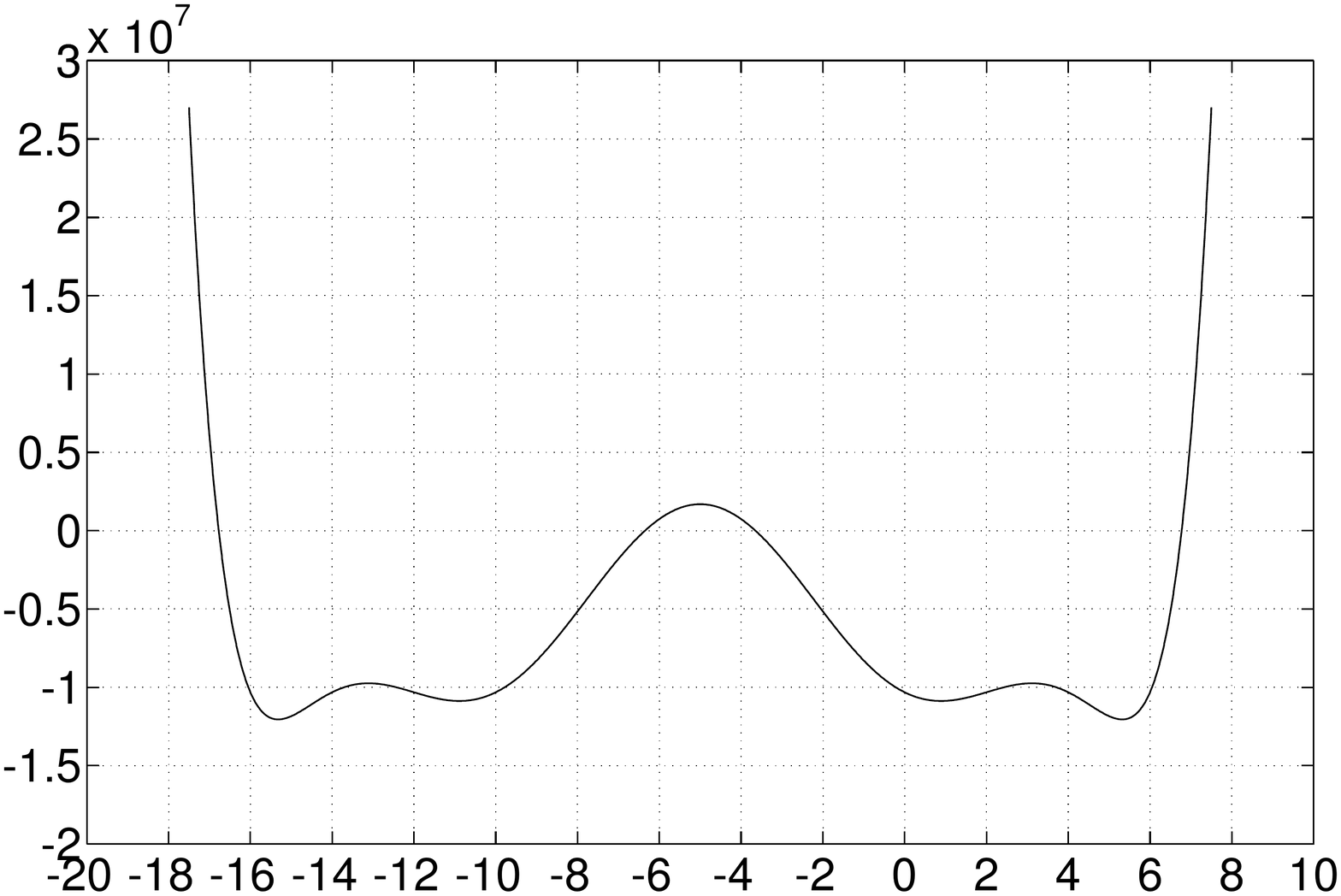}}
\caption{Graph of $P_m$ for $m=4$ and $n=n^*=18$ .}
\end{center}
\end{figure}

We recall that in the case $m=2$, the possibility of factorizing
$P_m$ as a product of four real polynomial of degree 1 was
fundamental for proving stability of radial entire solutions of
\eqref{eq_1} corresponding to the case $\beta \in \partial\mathcal
A_\alpha$, see the proofs of Theorem 6 and Lemma 12 in
\cite{bffg}. The impossibility for $m\ge 4$ even of having a
factorization of $P_m$ as a product of real polynomials of degree
1 also in dimensions $n\ge n^*$ makes difficult to understand if
the existence of stable solutions of \eqref{eq_1} occurs for such
dimensions.

\appendix
 \section*{Appendix}
\setcounter{section}{1}
 \setcounter{proposition}{0}

In this appendix we recall from \cite{fls} and
\cite{MacKennaReichel} a couple of results concerning solutions of
\eqref{Cauchy}: the first one is a result dealing with continuous
dependence on initial conditions and the second a comparison
principle.

\begin{proposition} \label{dipendenza-continua}
(\cite{fls}) For any $n\ge 1$ we have:

\begin{itemize}
\item[(i)] for any $\alpha_0,\dots,\alpha_{m-1}\in \R$ problem
\eqref{Cauchy} admits a unique local solution defined on the
maximal interval of continuation $[0,R)$ with $0<R\le +\infty$;

\item[(ii)] let $\alpha_0,\dots,\alpha_{m-1}\in\R$ and let
$\{\alpha_{0,k}\},\dots,\{\alpha_{m-1,k}\}$ sequences in $\R$ such
that
$$
\alpha_{0,k}\to \alpha_0 \, ,  \qquad \alpha_{i,k}\to \alpha_i
\quad \text{for any } i\in\{1,\dots,m-1\}\, , \qquad \text{as }
k\to +\infty \, .
$$
Denote by $u$ and $u_k$ the solutions of \eqref{Cauchy}
corresponding respectively to the initial values
$\alpha_0,\dots,\alpha_{m-1}$ and
$\alpha_{0,k},\dots,\alpha_{m-1,k}$. Denote by $[0,R)$, $0<R\le
+\infty$ the maximal interval of continuation of $u$. Then for any
$S\in (0,R)$ there exists $\overline k>0$ such that for any
$k>\overline k$, $u_k$ is well defined in $[0,S]$ and moreover
$u_k\to u$ uniformly in $[0,S]$ as $k\to +\infty$.
\end{itemize}
\end{proposition}

\begin{proposition} \label{mckenna reichel}
Assume that $f:\R\rightarrow \R$ is locally Lipschitz continuous
and monotonically increasing and let $m\in \N$, $m\ge 1$. Let
$u,v\in C^{2m}([0,R))$ be such that
\begin{equation} \label{inequalities}
\left \{ \begin{array}{ll} \forall r\in[0,R):
\Delta^m u(r)-f(u(r))\geq \Delta^m v(r)-f(v(r))\,, \\
u(0)\geq v(0)\,,\quad u'(0)= v'(0)=0\,, \\
\Delta^k u(0)\geq \Delta^k v(0)\,,\quad (\Delta^k u)'(0)=
(\Delta^k v)'(0)=0 \quad \text{for any } k=1,\dots, m-1 \,    .
\end{array}
\right.
\end{equation}
Then, for all $r\in[0,R)$ and for all $k\in \{1,\dots,m-1\}$ we
have
\begin{equation} \label{eq:comparison}
u(r)\geq v(r)\,,\quad u'(r)\geq v'(r)\,,\quad \Delta^k u(r)\geq
\Delta^k v(r)\,,\quad(\Delta^k u)'(r) \geq (\Delta^k v)'(r)\,.
\end{equation}
Moreover, the initial point $0$ can be replaced by any initial
point $\rho>0$ if all the four initial data are weakly ordered and
a strict inequality in one of the initial data at $\rho\geq 0$ or
in the differential inequality in $(\rho,R)$ implies a strict
ordering of $u,u',\Delta^k u,(\Delta^k u)'$ and $v,v',\Delta^k
v,(\Delta^k v)'$ on $(\rho,R)$ for any $k\in \{1,\dots,m-1\}$.
\end{proposition}

\begin{proof} Suppose first that at least one of the inequalities in the differential inequality or
in the initial conditions is strict. Then there exist $0<R_0\le R$
and $k\in \{1,\dots,m-1\}$ such that
\begin{equation} \label{ine_k}
\Delta^k u(r)>\Delta^k v(r) \qquad \text{for all } r\in(0,R_0) \,
.
\end{equation}
We may assume that $R_0$ is optimal in the sense that
$$
R_0=\sup\{r\in (0,R):\Delta^k u(s)>\Delta^k v(s) \ \text{for any }
s\in (0,r)\} \, .
$$
Suppose by contradiction that $R_0<R$.

Writing $\Delta^k u(r)=r^{1-n} (r^{n-1}(\Delta^{k-1} u)')'$ and
using the large inequalities in \eqref{inequalities}, two
successive integrations yield
\begin{equation*}
(\Delta^{k-1} u(r))'>(\Delta^{k-1} v(r))' \qquad \text{and} \qquad
\Delta^{k-1} u(r)>\Delta^{k-1} v(r) \qquad \text{for all }
r\in(0,R_0) \, .
\end{equation*}

After a finite number of steps we arrive to
\begin{equation*}
u(r)>v(r) \qquad \text{for all } r\in (0,R_0)
\end{equation*}
and hence
\begin{align*}
 (r^{n-1} (\Delta^{m-1}u(r))')'&=r^{n-1} \Delta^m u(r)
\ge r^{n-1}\left[f(u(r))+\Delta^m v(r)-f(v(r))\right] \\
& \ge r^{n-1} \Delta^m v(r)=(r^{n-1} (\Delta^{m-1}v(r))')' \quad
\text{for all } r\in (0,R_0) \, .
\end{align*}
We can restart the iterative procedure of successive integrations
to conclude that the map
$$
r\mapsto \Delta^k u(r)-\Delta^k v(r)
$$
is nondecreasing so that by \eqref{ine_k} we have $\Delta^{k}
u(r)>\Delta^{k} v(r)$ for all $r\in(0,R_0]$. The strict inequality
until $r=R_0$ contradicts the optimality of $R_0$. The validity of
the inequality \eqref{ine_k} in the whole interval $(0,R)$ yields
the strict inequalities in the whole $(0,R)$ also for the other
terms in \eqref{eq:comparison}.

It remains to prove the proposition in the case of large
inequalities. To this purpose we consider the unique solution
$u_\varepsilon$ of the initial value problem
\begin{equation}
\left \{ \begin{array}{ll}
\Delta^m u_\varepsilon-f(u_\epsilon)=\Delta^m u-f(u)  \\
u_\varepsilon(0)=u(0)+\varepsilon\, , \quad u'_\varepsilon(0)=0 \\
\Delta^k u_\varepsilon(0)=\Delta^k u(0) \, \quad (\Delta^k
u_\varepsilon)'(0)=(\Delta^k u)'(0) \qquad \text{for all }
k\in\{1,\dots,m-1\} \, .
\end{array}
\right.
\end{equation}
Existence and uniqueness for this initial value problem follows by
Proposition \ref{dipendenza-continua} (i).

From the first part of the proof we infer that for any $r\in
(0,R)$ and any $\varepsilon>0$
\begin{align*}
& u_\varepsilon(r)>v(r)\, , \quad u'_\varepsilon(r)>v'(r) \, , \\
& \Delta^k u_\varepsilon(r)>\Delta^k v(r) \, , \quad (\Delta^k
u_\varepsilon)'(r)>(\Delta^k v)'(r) \qquad \text{for any }
k\in\{1,\dots,m-1\}\, .
\end{align*}
Letting $\varepsilon\to 0^+$ and using Proposition
\ref{dipendenza-continua} (ii) we arrive to the conclusion. We
finally observe that the previous argument can be repeated by
replacing the initial condition at $r=0$ with an initial condition
at any other $\rho\in (0,R)$.
\end{proof}

{\bf Acknowledgement}

The authors are grateful to Elvise Berchio for the fruitful
discussions during the preparation of this paper.

\end{document}